\documentclass[11pt,leqno]{article}  
\usepackage{amsmath,amstext, amsthm,amsfonts,amssymb,amsthm}  
 \numberwithin{equation}{section} 
\def\R{\mathbb{ R}}  
  
\usepackage{epsfig}
\newtheorem{thm}{Theorem}[section]
\newtheorem{lem}[thm]{Lemma}
\renewcommand{\O}{\mathcal{O}}
\newtheorem{cor}[thm]{Corollary}
\newcommand{\E}{{\cal E}_q}

\newcommand{\be}{\begin{equation}}
\newcommand{\ee}{\end{equation}}

\newcommand{\ba}{\begin{array}}
\newcommand{\ea}{\end{array}}

\newcommand{\bg}{\begin{gathered}}
\newcommand{\eg}{\end{gathered}}
\renewcommand{\r}{\rho}

\renewcommand{\a}{\alpha}
\renewcommand{\b}{\beta}
\renewcommand{\t}{\theta}
\renewcommand{\O}{\mathcal{O}}
\renewcommand{\l}{\lambda}
\newcommand{\f}{\phi}
\renewcommand{\r}{\rho}

\newcommand{\z}{\zeta}

\newcommand{\mc}{\mathcal}
\newcommand{\cqHp }{continuous $q$-Hermite polynomials }

\newcommand{\aw}{ Askey--Wilson  }
\newcommand{\gauss}[2]{\genfrac{[}{]}{0pt}{}{#1}{#2}_q}

\newtheorem{rem}[thm]{Remark}
\newcommand{\bea}{\begin{eqnarray}}
\newcommand{\eea}{\end{eqnarray}}

\newcommand{\D}{\mathcal{D}_q}

\newcommand{\Sum}{\sum_{n=0}^\infty}
\renewcommand{\(}{\left( } \renewcommand{\)}{\right) }

\begin{document}
\title{$q$-Fractional Askey--Wilson Integrals and Related  Semigroups of Operators}

\author{ Mourad E. H.  Ismail 
 \\
 \and Ruiming Zhang \thanks{Corresponding author, research partially supported by National Natural 
 Science Foundation of China, grant No. 11771355.}
 \and Keru Zhou}
\maketitle


\begin{abstract}
 We introduce three one-parameter semigroups of operators and 
 determine their spectra. Two of them are fractional integrals 
 associated with the Askey--Wilson  operator. We also study these families  as families of positive linear approximation operators. Applications 
 include connection relations and bilinear formulas 
 for  the Askey--Wilson  polynomials. We also introduce a $q$-Gauss--Weierstrass transform and prove a representation and inversion theorem for it. 
\end{abstract}

 Filename: IZZFracV3.tex
 
 
 
 \section{Introduction} The Askey-Wilson operator $\D$ is a 
 divided difference   operator 
 which acts as a lowering (or annihilation) operator on the Askey--Wilson  polynomials and many special and limiting cases of them.  
 Several works defined a right inverse to the Askey--Wilson  operator, see for 
 example 
 \cite{Bro:Ism}, \cite{Ism:Zha1}, \cite{Ism:Rah}. The idea in these studies 
 was  to define $\D^{-1}$ on a basis for a weighted $L_2$ space then 
 extend it by linearity. On the other hand Hermann Weyl 
 introduced his fractional integral operator of order $\a$ as 
 the linear operator whose action on $e^{inx}, n \ne 0$ is 
 $e^{inx}/(in)^\a$. He then extended this by linearity, and
 identified the result as an integral transform. This was 
 done for the Askey-Wilson operator by Ismail in \cite{Ism2} and the 
 kernel was identified explicitly in \cite{Ism:Rah}. 
 
 The purpose of this work is to introduce a variation 
 on the Ismail-Rahman operators. This variation naturally leads to a semigroup of contraction operators, which we denoted by $\{T_a: a>0\}$.  This semigroup also has the property $\D T_a= T_{a-1}, a >1$. We study the applications of this semigroup as well as the semigroup of the adjoint operators $S_a$.  It turns out that  $S_a$ has the same parameter decreasing property with $\D$ replaced by a new operator $\mathcal{C}_q$.  Furthermore, we also identify their corresponding infinitesimal generators. 
   The operators $T_a$ are studied as approximation 
operators, the rate of convergence of $T_af$ to $f$ 
on $C[-1,1]$ is found, and  their 
adjoints $S_a$ are investigated as fractional  integral operators.  We note that there are 
many $q$-analogues of the standard operators in approximation 
theory but they all let $q \to1^-$ as their large parameter tends 
to $\infty$. Some references are in \cite{Ost}. The interested 
reader can find many references simply by 
searching the internet for ``$q$-Bernstein" 
or ``$q$-Sz\'{a}sz", $\cdots$". 

This work initiates a study of fractional powers of inverses of the Askey--Wilson  operator, so it is a $q$-analogue of the classical fractional integrals. 
Fractional integrals and derivatives have appeared in many areas of analysis and some of their applications are in Sneddon's classic 
\cite{Sne}.  It must be noted that the existing theory of $q$-fractional calculus deals with powers and inverses of the operator 
\begin{eqnarray}
\label{eqdq}
(D_qf)(x) = \frac{f(x)- f(qx)}{x-qx},
\end{eqnarray}
which is radically different from the Askey--Wilson operator defined 
in \eqref{eqDefAWOP}  below. Some references are \cite{Als:Ver1}, \cite{Als:Ver2},  and \cite{Ann:Man}.

We wish to indicate  that this work is just the starting point of a 
possibly rich theory of $q$-fractional integrals associated 
with Askey--Wilson type operators. Many problems remain open. For example we 
have not touched the problem of describing the range of the 
$q$-fractional integral operators 
defined in this work as operators on general weighted $L_p$ spaces.  
Examples of the work on describing the 
range of the classical fractional integrals  are \cite{Roo}, \cite{McB:Spr}, 
\cite{Raf:Sam}.  We described the range only when $T_a$, or its adjoint $S_a$ 
act on the weighted $L_2$ space of the $q$-Hermite polynomials.  In a future work we plan to apply the $q$-fractional 
integral operators presented here to solve dual  series 
equations. References to similar work for fractional integrals are in 
\cite{Sne}, see also Askey's interesting paper \cite{Ask}.  $q$-fractional integrals associated with $D_q$ of \eqref{eqdq} were 
used in \cite{Ash:Ism:Man} to solve dual and triple series equations. 
Another work in progress is to study analogues of Leibniz rule 
for the fractional 
powers of the Askey--Wilson  operators. Some progress has been made towards 
the Leibniz rule. The Leibniz rule for fractional powers of $D_q$ is
 in 
\cite{Als:Ver1}--\cite{Als:Ver2}. For details, see also \cite{Ann:Man}.  
 
Semigroups usually describe time evolutionary processes. It may be seen from the connection between semigroups and initial value problems, 
 \cite{But:Ber}. Given a semigroup $T(t)$ whose infinitesimal generator is $J$,  and a function $f$, independent of $t$, the function $(T(t) f)(x)$ solves the initial value problem $J y(x,t) = \frac{\partial y(x,t)}{\partial t}$, with initial condition $y(x,0) = f(x)$.  Therefore each infinitesimal generator we identify leads to a solution of an initial value problem. 
 
 We shall follow the standard notations for $q$-series as in 
 \cite{Gas:Rah}, \cite{And:Ask:Roy}.  All the results used in this work are in Ismail's book \cite{Ismbook}. We shall also use the 
 Rahman notation, \cite{Gas:Rah}, 
 \begin{eqnarray}
 h(\cos \t;a_1, \cdots, a_n) := \prod_{j=1}^n (a_j e^{i\t}, a_j e^{-i\t};q)_\infty.
 \end{eqnarray}
 
 The continuous $q$-Hermite polynomials play an essential role in the development of the $q$-fractional integral operators $T_a$.  Because of this we state few of their properties. They have the generating function 
 \begin{eqnarray}
 \Sum \frac{H_n(\cos \t|q)}{(q;q)_n} t ^n = \frac{1}{(te^{i\t}, te^{-i\t};q)_\infty},
 \label{eqGFHn}
 \end{eqnarray}
and the recurrence relation 
\begin{eqnarray}
\label{eqqHrr}
\bg
H_0(x|q) =1, \quad H_1(x|q) = 2x,\\
2xH_n(x|q) = H_{n+1}(x|q) + (1-q^n) H_{n-1}(x|q). 
 \eg
 \end{eqnarray}
 The normalized 
 weight function of the \cqHp is  given by
 \begin{eqnarray}
 w_H(x|q) = \frac{(q, e^{2i\t}, e^{-2i\t};q)_\infty}{2\pi \sqrt{1-x^2}}, 
 \end{eqnarray}
 then the orthogonality  relation is 
 \begin{eqnarray}
 \label{eqorcqH}
 \int_{-1}^1 H_m(x|q)H_n(x|q) w_H(x|q)  dx = (q;q)_n \delta_{m,n}.
 \end{eqnarray}
Details are in \cite{Ismbook}, \cite{And:Ask:Roy}, \cite{Koe:Swa}. 

Section 2 contains some preliminary material. In 
Section 3 we define the operators $T_a$ and we show that 
they form a  semigroup of strongly continuous operators. We also 
identify the infinitesimal generator of $T_a$. 
Section 4 evaluates the actions of the Askey--Wilson  operator on certain 
polynomial basis, which are $q$-analogues of $(x+a)^n$, and 
on $q$-exponential functions. In Section 5 we derive several 
properties of  the semigroup generated by $S_a$, the adjoints of $T_a$.
 We also include the inversion 
formulas for $T_a$ and $S_a$. In Section 6 we show that 
   $T_a$ maps an Askey--Wilson  polynomial to an 
   Askey--Wilson  polynomial with different parameters. This is 
   used to derive connection relations and bilinear formulas for 
   special  Askey--Wilson  polynomials. In Section 7 we study the 
   approximation properties of the operators $T_a$ as  maps from $C[-1,1]$ to infinitely differentiable functions.  We must note that the infinitesimal 
   generator of the semigroup is related to a  Voronovskaya 
   type relation. In Section 8, we prove that  $T_a$ are 
   contraction maps. In Section 9 we study 
   a family of operators $F_a$ initially introduced in \cite{Sus}. We 
   study their approximation properties as positive linear 
   approximation operators and identify their infinitesimal generator. 
   We also show that they are also essentially fractional powers 
   of an inverse 
   to a divided difference operator $B$ defined in 
   \eqref{eqdefB}, which is 
   not the \aw operator.  We also  indicate a transform for which 
   $F_a$ acts as a $q$-analogue of a multiplier. 
In Section 10 we briefly mention an analogue of the 
Gauss--Weierstrass transform and give a representation 
and inversion formula for our transform. In the last section, Section 11, 
we apply the fractional integral operators of Section 3 to solve certain dual integral equations following techniques developed earlier by Noble \cite{Nob} and Sneddon \cite{Sne}. 

The Askey-Wilson operator is a linear operator for functions on the real line that is defined via a complex variable $z$ and its reciprocal  $1/z$ on the complex plane. For any function $f(x)$ we write $x$ as 
$x = (z+1/z)/2$ and associate it to  the 
function $\breve{f}(z):= f(x)$.  Then the  \aw operator is defined by 
\bea
\label{eqDefAWOP}
(\D f)(x) = \frac{\breve{f}(q^{1/2}z) - \breve{f}(q^{-1/2}z)}{(q^{1/2}-q^{-1/2})(z-1/z)/2}.
\eea
 In the sequel we will assume that the functions are defined on 
 $ [-1,1]$ unless explicitly stated otherwise. In order to define $\D$ on such function, say $f$, we 
 must tacitly assume that $f$ is defined for $|z| \le q^{-1/2}$, with 
 $x = (z+1/z)/2$. For the brevity we will not always remind readers of this assumption but they must be made aware of it. 
 
 We wish to note that the $q$-Hermite polynomials play a central part 
 in the defining the semigroups of this paper. When $q >1$, with 
 appropriate scaling  the polynomials become the $q^{-1}$-Hermite 
 polynomials whose analytic properties are radically different, 
 \cite{Ism:Mas}. They 
 are now orthogonal on $\R$ with respect to infinitely many  probability 
 measures. The case $q > 1$ has a theory that parallels the theory presented here,  albeit subtler, and it will be developed in a future work.

 \section{Preliminaries}  
 
 Ismail and Rahman \cite{Ism:Rah} showed that the operator
$\D^{-1}$ defined by 
 \begin{eqnarray}
 \label{eq2.1}
 \bg 
\frac{2\; q^{1/4}} {(1-q)(q;q)_\infty}  (\D^{-1}f)(\cos \t) \\
 = 
 \int_0^\pi  
 \frac{(-q^{1/4} e^{i\t}, -q^{1/4} e^{-i\t};q^{1/2})_\infty w_H(\cos \f) 
 f(\cos \f)}
 {(-q^{1/4} e^{i\f}, -q^{1/4} e^{-i\f};q^{1/2})_\infty  h(\cos \f; 
 q^{1/2}e^{i\t}, q^{1/2} e^{-i\t})} \sin \f d\f, 
 \eg
 \end{eqnarray}
 is a right inverse to the Askey--Wilson  operator.  
 The original definition in \cite{Ism:Rah} had an additional $f$ 
 dependent constant term which we decided to drop.  
 What is important is that the strong operator limit
 $\lim_{p \to q^-}  \mathcal{D}_p \D^{-1}$ is still the identity operator. 
 The Ismail--Rahman approach follows the original approach of Hermann Weyl when he introduced the Weyl fractional integrals, 
 see \cite{Zyg}. 
 
 We now iterate $\D^{-1}f$ and compute $\D^{-n}$ for any 
 natural number $n$. 
 Recall that the Poisson kernel for the \cqHp is \cite[Theorem 13.1.6]{Ismbook}
 \begin{eqnarray}
 \label{eqPoiKer}
 \Sum H_n(\cos \t|q)  H_n(\cos \f|q) \frac{t^n}{(q;q)_n} = 
 \frac{(t^2;q)_\infty}{(t e^{i(\t+\f)}, t e^{i(\t-\f)}, t e^{-i(\t+\f)}, 
 t e^{i(\f-\t)};q)_\infty}. 
 \end{eqnarray}
     
 We shall need the following fact.
 
 \begin{thm}\label{Lem1}Let $p_n$ be orthonormal with respect to a 
 probability measure $\mu$ on a compact  interval $[a,b]$. If the 
 Poisson kernel  $P_r(x,y):= \Sum p_n(x)p_n(y) r^n$ is $\ge 0$ for all $x, y \in$ supp$(\mu) \subset  [a,b]$,  then 
 \begin{eqnarray}
 \lim_{r \to 1^-} \int_a^b  P_r(x,y) f(y) d\mu(y) = f(x)
 \notag
 \end{eqnarray}
is uniform  for all continuous functions on $[a,b]$.
 \end{thm}
 This follows from Korovkin's theorem, \cite{Kor}.  For completeness  
 we state Korovkin's theorem. Let $e_j(x) =x^j$. if ${L_n}$ is 
 a sequence of positive linear operators acting on continuous functions 
 on a compact interval $[a,b]$. If  $(L_n e_j)(x) \to e_j(x)$
 uniformly on $[a,b]$ for $j=0,1,2$ then $(L_nf)(x) \to f(x)$ uniformly 
 on $[a,b]$ for all continuous functions $f$. 
 
 \section{$q$-Fractional Integrals} 
 
 Motivated by the Ismail-Rahman operator \eqref{eq2.1} we define the operators
 \begin{eqnarray}
 \label{eqdefTa}
 \bg
 (T_a  f)(\cos \t) = \frac{(1-q)^a}{2^aq^{a/4}}\; (q^a;q)_\infty \\
 \times \int_0^\pi  
 \frac{(-q^{1/4} e^{i\t}, -q^{1/4} e^{-i\t};q^{1/2})_\infty w_H(\cos \f|q) f(\cos \f)}
 {(-q^{1/4} e^{i\f}, -q^{1/4} e^{-i\f};q^{1/2})_\infty  h(\cos \f; q^{a/2}e^{i\t}, q^{a/2} e^{-i\t})} \sin \f d\f. 
 \eg
 \end{eqnarray}
 
 \begin{thm}\label{thm1}
 The  operators  $\{T_a: a >0\}$ have the following properties. 
 
\noindent  $\textup{(a)}$ The family $\{T_a: a >0\}$ has the property, 
 $T_a T_b = T_{a+b}$.

\noindent   $\textup{(b)}$ On $C[-1,1]$, $T_a$ tends to the identity operator 
as $a \to 0^+$. 

\noindent $\textup{(c)}$  For $a>1$ we have the property $\D T_a = T_{a-1}$. 
 \end{thm}

 \begin{proof}[Proof of $\textup{(a)}$]  It is clear that 
 \begin{eqnarray}
 \notag
 \bg
\left(\frac{2q^{1/4}}{1-q} \right)^{a+b} \frac{(T_a T_bf)(\cos \t)}
{(-q^{1/4} e^{i\t}, -q^{1/4} e^{-i\t};q^{1/2})_\infty } \\
=  (q^a, q^b; q)_\infty \int_{[0,\pi]^2} 
 \frac{ w_H(\cos \f|q)}
 {h(\cos \f; q^{a/2} e^{i\t}, q^{a/2} e^{-i\t})}      \\
 \times  \frac{ w_H(\cos \psi|q) 
 f(\cos \psi) \sin \f \sin \psi } {(-q^{1/4} e^{i\psi}, 
 -q^{1/4} e^{-i\psi};q^{1/2})_\infty
 h(\cos \psi; q^{b/2}e^{i\f}, q^{b/2} e^{-i\f})} \;d\f d\psi.
 \eg
 \end{eqnarray}
 The Poisson kernel \eqref{eqPoiKer} and the completeness 
 of the \cqHp imply 
 \begin{eqnarray}
 \notag
 \bg
 \int_0^\pi \frac{(q^a, q^b; q)_\infty w_H(\cos \f|q)\sin \f \; d\f}
 {h(\cos \psi; q^{b/2}e^{i\f}, q^{b/2} e^{-i\f})
 (\cos \f; q^{a/2} e^{i\t}, q^{a/2} e^{-i\t})} \\
 =  \frac{(q^{a+b};q)_\infty} {h(\cos \psi; q^{(a+b)/2}e^{i\t}, 
 q^{(a+b)/2} e^{-i\t})}
 \eg
 \end{eqnarray}
 and the proof of part (a) is complete.
 \end{proof}
 
  \begin{proof}[Proof of $\textup{(b)}$] For continuous functions 
  $f$, $f(\cos \t)/(-q^{1/4} e^{i\t}, -q^{1/4} e^{-i\t};q^{1/2})_\infty$ is continuous, hence Theorem \ref{Lem1} implies part (b). 
  \end{proof}
  
 The proof of  (c) follows by direct computation.

 It must be noted that the kernel of the integral operator $T_a$ becomes quotients of products of theta functions, \cite{Whi:Wat}, 
 \cite{Rademacher}, when $a = 1$, that is the case of $\D^{-1}$ 
 in \eqref{eq2.1}. So this case is very special in many ways. 
 
 The next theorem describes the eigenvalues of $T_a$ as an 
 operator on $L_2[-1,1, w_H]$. 
 \begin{thm}\label{thm2}
 The only eigenfunctions of $T_a$ as operator $T_a$ acting on 
 $L_2[-1,1, w_H]$ are the functions $(-q^{1/4} e^{i\t}, 
 -q^{1/4} e^{-i\t};q^{1/2})_\infty 
 H_m(\cos \t\vert q)$ with  eigenvalues 
 \begin{eqnarray}
 \label{eqeigenvalues}
 \l_m = \frac{(1-q)^a}{2^aq^{a/4}} q^{ma/2}.
 \end{eqnarray}  
 \end{thm}
 \begin{proof}
 It is clear that the action of $T_a$ on $(-q^{1/4} e^{i\t}, 
 -q^{1/4} e^{-i\t};q^{1/2})_\infty H_m(\cos \t\vert q)$ is given by 
 \begin{eqnarray}
 \notag
 \bg
 \frac{(1-q)^a}{2^aq^{a/4}}\; (q^a;q)_\infty   \int_0^\pi  
 \frac{(-q^{1/4} e^{i\t}, -q^{1/4} e^{-i\t};q^{1/2})_\infty w_H(\cos \f|q)} 
 {h(\cos \f; q^{a/2}e^{i\t}, q^{a/2} e^{-i\t})} H_m(\cos \f\vert q) \sin \f d\f \\
 =\frac{(1-q)^a}{2^aq^{a/4}} (-q^{1/4} e^{i\t}, -q^{1/4} e^{-i\t};q^{1/2})_\infty \;   \int_{-1}^1 H_m(y|q) 
 w_H(y|q)
 \Sum H_n(x|q)H_n(y|q) \frac{q^{qn/2}}{(q;q)_n} \;dy \\
 = \frac{(1-q)^a}{2^aq^{a/4}} q^{am/2} (-q^{1/4} e^{i\t}, -q^{1/4} e^{-i\t};q^{1/2})_\infty \; H_m(\cos \t\vert q),
 \eg
 \end{eqnarray}
 where we used the orthogonality relation \eqref{eqorcqH} and the 
 Poisson kernel \eqref{eqPoiKer}. This shows that 
 $(-q^{1/4} e^{i\t}, -q^{1/4} e^{-i\t};q^{1/2})_\infty H_m(\cos \t\vert q)$  
 are eigenfunctions. 
 
 It is clear that the function on the right-hand side 
 of \eqref{eqPoiKer} is in 
$L_2[-1,1, w_H]$.  Let $f$ be an eigenfunction 
 with an eignenvalue $\l$. Since  
 $1/(-q^{1/4} e^{i\t}, -q^{1/4} e^{-i\t};q^{1/2})_\infty$ is bounded 
 above  and 
 below by positive numbers,  then $f/(-q^{1/4} e^{i\t}, -q^{1/4} e^{-i\t};q^{1/2})_\infty$ is in 
 $L_2[-1,1, w_H]$, so we set 
 \begin{eqnarray}
 \notag
\frac{ f(\cos \t)}{ (-q^{1/4} e^{i\t}, -q^{1/4} e^{-i\t};q^{1/2})_\infty} 
= \Sum f_n H_n(\cos \t|q).
 \end{eqnarray} 
 The orthogonality and completeness of the \cqHp system, 
 Parseval's theorem and the Poisson kernel \eqref{eqPoiKer}  show that 
 \begin{eqnarray}
 \notag
 \bg
\frac{ \l f(\cos \t)}{ (-q^{1/4} e^{i\t}, -q^{1/4} e^{-i\t};q^{1/2})_\infty} = 
\frac{T_a f(\cos \t)}{(-q^{1/4} e^{i\t}, -q^{1/4} e^{-i\t};q^{1/2})_\infty} \\
 = \frac{(1-q)^a}{2^aq^{a/4}}
 \int_0^\pi  \Sum f_n H_n(\cos \t|q)  \sum_{k=0}^\infty H_k(\cos \t|q) 
 H_k(\cos \f|q) \frac{q^{ka/2}}{(q;q)_k}   \;   w_H(\cos \f|q)   
   \sin \f d\f\\
   = \frac{(1-q)^a}{2^aq^{a/4}}  \Sum f_n   q^{na/2}H_n(\cos \t|q).       . 
 \eg
 \end{eqnarray}
 This shows that the only eigenfunctions are  $(-q^{1/4} e^{i\t}, 
 -q^{1/4} e^{-i\t};q^{1/2})_\infty H_m(\cos \t\vert q)$ and that the 
 corresponding eigenvalues are 
 given by   \eqref{eqeigenvalues}. 
 \end{proof}
 
The operator $T_a$ is the limit of finite rank operators (its restriction 
to the span of $\{H_k(x|q):0 \le k \le n\}$).     
Therefore $T_a$ is compact for any $a >0$ and Hilbert space $L_2[-1,1,w_H]$ is infinite dimensional, and we conclude that the operator $T_a$ is not invertible for any $a >0$.  This is also implied  by the fact that $\l_n \to 0$, hence $\l =0$ is in the spectrum.

  \begin{rem}
  	\label{analytic-continuation}
  	We have the following general comments for $T_a$, Similar remarks apply  to $S_a, F_a$ in later sections:
  	\begin{enumerate} 
  		\item   $T_a$ is defined for all $\Re(a)>0$ in \eqref{eqdefTa}, and parts \rm{(a) and (c)} of Theorem \ref{thm1} remain valid, but $T_a$  may not be positive anymore. 
  \item Let 
  \bea
  \label{eqdefg}
  g(\cos \t) = (-q^{1/4} e^{i\t}, -q^{1/4} e^{-i\t};q^{1/2})_\infty.
  \eea
  For each $f\in L_2[-1,1,w_H]$,  let 
  \bea
  \label{eqdefL}
  Lf=gf, \ L^{-1}f=f/g.
  \eea 
  Since both $g,1/g$ are continuous on $[-1,1]$,  then the operators  $L^{-1}, L$,  are bounded on $L_2[-1,1,w_H]$. By Theorem \ref{thm2}, for  $\Re(a)>0$ the operator 
  		$L^{-1}\, T_a \, L$ has eigen-system $\{\lambda_m, H_m (x\vert q)\}$, hence it is a trace class
  		operator. Consequently, $\{T_a :\Re(a)>0\}$ are trace class operators since trace class is a two sided ideal.
  		\item  As matter of fact, $T_a$ can be defined for each  $a\in \mathbb{C}$  satisfying  $q^{\Re(a)/2+n}\neq 1$ for  $n=0,1,...$.  The obtained operator is still compact since its kernel is continuous on $[-1,1]\times[-1,1]$.  As a result of compactness, none of  $T_a$s is  invertible since the underlying Hilbert space $L_2[-1,1,w_H]$ is infinitely dimensional. By applying \rm{(c)} of Theorem \ref{thm1} 
		we can extend the semigroup law \rm{(a)} to complex cases as far as all the operators involved exists. 
  		\item Another immediate consequence of Theorem \ref{thm1} is that $\D$ is  
  		$T_{-1}$, in the sense that the strong operator limit  $\lim_{a\to 1^+} \D T_a$ 
  		is the identity operator $I$.
  	\end{enumerate} 
  \end{rem}
  We now identify the infinitesimal generator  of $T_a$. From Theorem 
 \ref{thm2} it follows that 
 \begin{eqnarray}
 T_a \left(g(x)H_m(x|q)\right) 
 = \left(\frac{1-q}{2} q^{(2m-1)/4}\right)^a g(x)H_m(x|q)
 \end{eqnarray}
  where $g$ is defined in \eqref{eqdefg}.
  Thus, with $a \to 0$, $I$ the identity operator, we have 
 \begin{eqnarray}
 \notag
  (T_a- I) g(x)H_m(x|q) = a\log\left(\frac{1-q}{2} q^{(2m-1)/4}\right) 
  g(x)H_m(x|q) +R_a, 
  \end{eqnarray}
  where the $L_2$-norm of $R_a$ is of $\O(a^2 m^2 \left \Vert g H_m(\cdot \vert q)\right \Vert)$.
 Then the linear operator $J=\lim _{a\to0} a^{-1}[T_a- I]$ is defined on all finite linear combinations of the polynomials $\{g(x)H_n(x\vert q)\}$. Let $gf \in L_2[-1,1, w_H]$ with $g$ as in \eqref{eqdefg},  since the $q$-Hermite polynomials form an orthogonal basis, then 
  $gf = \Sum c_nH_n(x|q)$. 
  For all $f \in L_2[-1,1, w_H]$  such that
  \begin{eqnarray}
  \label{domain-J}
  f /g= \Sum c_nH_n(x|q),  \\
  \sum _{m=0}^{\infty }m^2 |c_m|^2 (q;q)_m <\infty,  
  \end{eqnarray} 
 it is clear that  the infinitesimal  generator $J$  is defined by 
  \begin{eqnarray}
  \left(Jf\right)(x)  = g(x)\Sum \log\left(\frac{1-q}{2} q^{(2m-1)/4}\right) c_m H_m(x|q). 
  \end{eqnarray}
Now both $g$ and $1/g$ are continuous and uniformly bounded on $[-1,1]$, hence all $f$ satisfy \eqref{domain-J}  form a dense subspace of  $L_2([-1,1], w_H)$. 

 \section{$q$-Analogues of Powers and Exponential Functions}
In \cite{Ism:Sta1} and \cite{Ism:Sta2} Ismail and Stanton  introduced the bases 
 \begin{eqnarray} \label{eqDeffn&rn}
 \bg
 \phi_n(x)  =  (q^{1/4}e^{i\t}, q^{1/4} e^{-i\t};q^{1/2})_n =
 \prod_{k=0}^{n-1} (1- 2xq^{1/4+k/2} + q^{1/2+ k}), \label{eqdefphin1/4}\\
 \rho_n(x) = (1+e^{2i\t}) e^{-in \t}(-q^{2-n} e^{2i\t};q^2)_{n-1}, n>0, \quad \r_0(x):=1,
 \eg
 \end{eqnarray}
 and established $q$-Taylor series expansions of entire functions in these bases. 
 
 It is clear that  we can replace $n$ by a general parameter 
 $\b$ in \eqref{eqDeffn&rn}.
 
 \begin{thm} The action of $T_a$ on  $ \phi_\b(x)$ is give by 
 \begin{eqnarray}
 \label{eqTaobphi}
  T_a\phi_\b(x) = \frac{(1-q)^a}{2^aq^{a/4}}  \frac{(q^{a+\b+1};q)_\infty}{(q^{\b+1};q)_\infty}\;  \phi_{\a+\b}(x). 
 \end{eqnarray} 
 \end{thm}
 \begin{proof}
It is clear that  $T_a\phi_\b(\cos \t)$ equals 
\begin{eqnarray}
\notag
\bg
\frac{(1-q)^a}{2^aq^{a/4}}\; (q^a;q)_\infty \\
\times  \int_0^\pi  
 \frac{(-q^{1/4} e^{i\t}, -q^{1/4} e^{-i\t};q^{1/2})_\infty w_H(\cos \f|q)\; \sin \f \;d\f}
 {(-q^{\b/2+1/4} e^{i\f}, -q^{\b/2+1/4} e^{-i\f};q^{1/2})_\infty  h(\cos \f; q^{a/2}e^{i\t}, q^{a/2} e^{-i\t})} \\
= \frac{(1-q)^a}{2^aq^{a/4}}\; (q^a;q)_\infty \\
\times  \int_0^\pi  
 \frac{(-q^{1/4} e^{i\t}, -q^{1/4} e^{-i\t};q^{1/2})_\infty w_H(\cos \f|q)\;  \sin \f \;d\f}
 {h(\cos \f; -q^{\b/2+1/4}, -q^{\b/2+3/4}, q^{a/2}e^{i\t}, q^{a/2} e^{-i\t})}.
 \eg 
\end{eqnarray}
The integral is an Askey-Wilson integral and using the evaluation \cite{And:Ask:Roy}, \cite{Gas:Rah}
\begin{eqnarray}
 \int_0^\pi  
 \frac{w_H(\cos \f|q)}
 {h(\cos \f; a_a,a_2, a_3, a_4)} \sin \f d\f = 
 \frac{(a_1a_2a_3a_4;q)_\infty}{\prod_{1 \le j < k \le 4} 
 (a_ja_k;q)_\infty}.
\end{eqnarray}
 \end{proof}
 We note that \eqref{eqTaobphi} is an analogue of the action if a fractional integral operator on $x^\b$. 
 
 Formula \eqref{eqTaobphi} is very important because we know how to expand entire functions as $\Sum c_n \f_n(x)$. Indeed 
if $f(x)= \Sum f_n \f_n(x)$, then 
 \begin{eqnarray}
T_a \Sum f_n \frac{\f_n(x)}{(q;q)_n} 
=  \frac{(1-q)^a}{2^aq^{a/4}} 
 \Sum  f_n \; \frac{ \phi_{\a+n}(x)}
 {(q;q)_{n+a}}.
 \end{eqnarray}

 Recall the definition of the 
 $q$-exponential function $\E$, which we introduced in \cite{Ism:Zha1},
\begin{eqnarray}
\label{eqqExp}
\bg
\quad  \mc{E}_q(\cos\theta;\alpha) =\frac{\(\alpha^2;q^2\)_\infty}
{\(q\alpha^2;q^2\)_\infty} 
 \Sum\(-ie^{i\theta}q^{(1-n)/2},-ie^{-i\theta}q^{(1-n)/2};q\)_n 
 \frac{(-i\alpha)^n}{(q;q)_n}\,q^{n^2/4}. 
\eg
\end{eqnarray}
In \cite{Ism:Zha1}, it was  shown that, 
\begin{eqnarray}
\label{eqqExinqH}
(qt^2;q^2)_\infty\E(x;t)= \Sum\frac{q^{n^2/4}t^n}{(q;q)_n}H_n(x\,|\, q). 
\end{eqnarray}  
We note that $\E(x;t)$ is a $q$-analogue of $e^{xt}$. 

\begin{thm}
The operators $T_a$ have the properties
\begin{eqnarray}\label{eqTaCurlyE}
\bg
T_a [(-q^{1/4} e^{i\t}, -q^{1/4} e^{-i\t};q^{1/2})_\infty \E(\cos \t;t)]\\
 = \frac{(1-q)^a(q^{a+1}t^2;q^2)_\infty}{2^aq^{a/4}(qt^2;q^2)_\infty} 
 (-q^{1/4} e^{i\t}, -q^{1/4} e^{-i\t};q^{1/2})_\infty  \; \E(\cos \t; q^{a/2}t) 
 \eg
\end{eqnarray}
and 
\begin{eqnarray}
\bg
T_a \left[\frac{(-q^{1/4} e^{i\t}, -q^{1/4} e^{-i\t};q^{1/2})_\infty} 
{(t e^{i\t}, t e^{-i\t};q)_\infty}\right] 
= \frac{(1-q)^a}{2^a q^{a/4}}
\frac{(-q^{1/4} e^{i\t}, -q^{1/4} e^{-i\t};q^{1/2})_\infty} 
{(t q^{a/2}e^{i\t}, t q^{a/2}e^{-i\t};q)_\infty}.
\eg
\end{eqnarray}
\end{thm}
 We omit the proof, which uses the orthogonality of the \cqHp and the 
 Poisson kernel and is very similar to the proofs in Section 2. 
 
 Ismail, Stanton, and Viennot \cite{Ism:Sta:Vie} proved that 
 \begin{eqnarray}
 \int_0^\pi \frac{w_H(\cos \t)  \sin \t }{h(\cos \t;  a_1, \cdots, a_k)}\; d\t, 
 \end{eqnarray}
 is essentially the generating function of the crossing numbers of perfect matchings of sets with $ a_1, \cdots, a_k$ as generating function variables. This enables us to evaluate the integrals
 \begin{eqnarray}
 \notag
\bg
T_a \left[\frac{(-q^{1/4} e^{i\t}, -q^{1/4} e^{-i\t};q^{1/2})_\infty} 
{(t e^{i\t}, t e^{-i\t};q)_\infty h(\cos \f; a_1, \cdots, a_k)}\right] 
\eg
\end{eqnarray}
 as a power series in $ a_1, \cdots, a_k$.

Formula \eqref{eqTaCurlyE} has the following  curious implication, 
\bea
\label{eqIntertwine}
\bg
\int_{-1}^1\frac{ \E(x;t) (T_af)(x) w_H(x|q)}{(-q^{1/4} e^{i\t}, -q^{1/4} e^{-i\t};q^{1/2})_\infty}\; dx \qquad \\
= \frac{(1-q)^a\; (q^{a+1}t^2;q^2)_\infty}{2^aq^{a/4}\; (qt^2;q^2)_\infty}
 \int_{-1}^1 \frac{\E(x,tq^{a/2}) f(x) w_H(x|q)}{(-q^{1/4} e^{i\t}, -q^{1/4} e^{-i\t};q^{1/2})_\infty}\, dx.
\eg
\eea
\begin{proof}
The left-hand side of \eqref{eqIntertwine} is 
\bea
\notag
\bg
 \frac{(1-q)^a}{2^aq^{a/4}}\; (q^a;q)_\infty \int_0^\pi \int_0^\pi  
 \E(\cos \t;t) 
 w_H(\cos \t|q) \sin \t \\
 \times 
 \frac{ w_H(\cos \f|q)\; f(\cos \f)\; \sin \f}
 {(-q^{1/4} e^{i\f}, -q^{1/4} e^{-i\f};q^{1/2})_\infty  h(\cos \f; q^{a/2}e^{i\t}, q^{a/2} e^{-i\t})} \; d\f d\t \\
 =   \int_0^\pi \frac{ f(\cos \f)\,  w_H(\cos \f|q)  \sin \f}{
(-q^{1/4} e^{i\f}, -q^{1/4} e^{-i\f};q^{1/2})^2_\infty} \\
\times 
(T_a  (-q^{1/4} e^{it}, -q^{1/4} e^{-it};q^{1/2})_\infty \E(\cos \t;t))
(\cos \f)  d\f \\
= \frac{(1-q)^a\; (q^{a+1}t^2;q^2)_\infty}{2^aq^{a/4}\; (qt^2;q^2)_\infty} 
\int_{-1}^1  \frac{\E(x,tq^{a/2}) f(x) w_H(x|q)}{(-q^{1/4} e^{i\t}, -q^{1/4} e^{-i\t};q^{1/2})_\infty}\, dx. 
\eg
\eea
This completes the proof. 
\end{proof}
 
 We shall revisit formulas of the type \eqref{eqIntertwine} in 
 Section 9.

 \section{Adjoints and Inversion} We realized that the Riemann-Liouville and the 
 Weyl fractional integral operators are adjoints with respect to the inner 
 product $(f,g) = \int_0^\infty f(x) \overline{g(x)} \, dx$.  
 
 We define the inner product 
 \begin{eqnarray}
(f,g) =  \int_{-1}^1 f(x) \overline{g(x)} w_H(x|q) dx.  
 \end{eqnarray}
 It is easy to see that the adjoint of $T_a$ is given by
  \begin{eqnarray}\notag 
 \bg
 (T^*_a  f)(\cos \t) = \frac{(1-q)^a}{2^aq^{a/4}}\; (q^a;q)_\infty \\
 \times \int_0^\pi  
 \frac{(-q^{1/4} e^{i\f}, -q^{1/4} e^{-i\f};q^{1/2})_\infty w_H(\cos \f|q) 
 f(\cos \f)}
 {(-q^{1/4} e^{i\t}, -q^{1/4} e^{-i\t};q^{1/2})_\infty  
 h(\cos \f; q^{a/2}e^{i\t}, q^{a/2} e^{-i\t})} \sin \f \; d\f, 
 \eg
 \end{eqnarray}
 where we have applied the symmetry 
\begin{equation} \notag
h(\cos\theta;q^{a/2}e^{i\phi},q^{a/2}e^{-i\phi})=h(\cos\phi;q^{a/2}e^{i\theta},q^{a/2}e^{-i\theta}).\label{eq:5.3}
\end{equation}
 We shall use the operators $S_a$,
  \begin{eqnarray}
 \bg
 (S_a  f)(\cos \t) = \frac{(1-q)^a}{2^a q^{a/4}}\; (q^a;q)_\infty \\
 \times \int_0^\pi  
 \frac{(-q^{1/4} e^{i\f}, -q^{1/4} e^{-i\f};q^{1/2})_\infty w_H(\cos \f|q) 
 f(\cos \f)}
 {(-q^{1/4} e^{i\t}, -q^{1/4} e^{-i\t};q^{1/2})_\infty  
 h(\cos \f; q^{a/2}e^{i\t}, q^{a/2} e^{-i\t})} \sin \f \; d\f. 
 \eg
 \end{eqnarray}
 
  \begin{thm}\label{thm4}
 The family of operators  $\{S_a: a >0\}$ has the following properties. 
 
\noindent  $\textup{(a)}$ $\{S_a: a >0\}$ is  a semigroup.

\noindent  $\textup{(b)}$ $S_a$ tends to the identity operator 
as $a \to 0^+$. 

\noindent $\textup{(c)}$  For $a>1$ we have the property  $\mathcal{C}_q S_a = S_{a-1}$  where
$\mathcal{C}_q$ is defined by
\bea
\label{c_q}
(\mathcal{C}_q\;f)(x)=2\frac{\breve{f}(q^{1/2}z) -z^4 \breve{f}(q^{-1/2}z)}{(1-q)(1-z^2)z}.
\eea
 \end{thm}
 \begin{proof}
It is clear that $T^*_a$ and $T^*_b$ commute since $T_a$ and  $T_b$ 
commute. More over $S_a S_b = S_{a+b}$ follows from 
part (a) of Theorem \ref{thm1} and the property of adjoints.  The 
remaining parts can be proved similar to the proofs of the corresponding parts in 
Theorem \ref{thm1}.   
 \end{proof} 
 
  One can similarly prove Theorem \ref{thm5} below. 
  \begin{thm}\label{thm5}
 The only eigenfunctions of $S_a$ as operator $S_a$ acting on 
 $L_2[-1,1, w_H]$ are the functions 
 $H_m(\cos \t\vert q)/(-q^{1/4} e^{i\t}, 
 -q^{1/4} e^{-i\t};q^{1/2})_\infty$ with  eigenvalues 
 \begin{eqnarray}
 \label{eqeigenvaluesS}
 \l_m = \frac{(1-q)^a}{2^aq^{a/4}} q^{ma/2}.
 \end{eqnarray}  
 Moreover the eigenfunctions form a basis for $L_2 [-1,1,w_H]$. 
 \end{thm}

 Let $g$ be as in \eqref{eqdefg} and 
for any $f\in L_2 [-1,1,w_H]$, let
\begin{equation}
(Kf)(\cos\phi)=\frac{f(\cos\phi)}{g^{2}(\cos\phi)}.\label{eq:5.4}
\end{equation}
It is clear that $K$ is invertible and 
\begin{equation}
(K^{-1}f)(\cos\phi)=g^{2}(\cos\phi)f(\cos\phi).\label{eq:5.5}
\end{equation}
 It is clear that 
 \bea
 \label{eqS=kT}
 S_a =   K \;T_a \; K^{-1}.
 \eea
\begin{thm}
Let $J_{t}$ and  $J_{s}$ be the infinitesimal generators of $T_{a}$ and    
  $S_{a}$, respectively  Then
\begin{equation}
J_{s}=K\; J_{t}\;K^{-1}.     
\label{eq:5.7}
\end{equation}
 Moreover, for any $h\in L_{2}[-1,1,w_{H}]$ such that 
\begin{equation}
g h=\sum_{n=0}^{\infty}c_{n}H_{n}(x\vert q),\quad 
\textup{and} \quad \sum_{n=1}^{\infty}n^{2}\left|c_{n}\right|^{2}(q;q)_{n}<\infty,\label{eq:5.8}
\end{equation}
 then
\begin{equation}
\left(J_{s}h\right)(x)=\sum_{n=0}^{\infty}\log\left(\frac{1-q}{2}q^{(2n-1)/4}\right)c_{n}\frac{H_{n}(x\vert q)}{g(x)}.\label{eq:5.9}
\end{equation}
\end{thm}

\begin{proof}
Since $g(x),\ 1/g(x)$ are continuous on $[-1,1]$, then for any $m\in\mathbb{Z}$
we have
\[
g^{m}\left(L_2[-1,1,w_{H}]\right)=L_2[-1,1,w_{H}].
\]
Then itt is clear that
\[
J_{s}=K\; J_{t}\; K^{-1}.
\]
 Now take  $f \in L_2[-1,1,w_{H}]$ where 
\[
\left(gf\right)(x)=\sum_{n=0}^{\infty}c_{n}H_{n}(x\vert q),\quad\textup{and}\quad  \sum_{n=1}^{\infty}n^{2}\left|c_{n}\right|^{2}(q;q)_{n}<\infty
\]
hold. 
 Since
\[
J_{t}\left(gf\right)(x)=g(x)\sum_{n=0}^{\infty}\log\left(\frac{1-q}{2}q^{(2n-1)/4}\right)c_{n}H_{n}(x),
\]
then 
\[
J_{s}\left(f/g\right)=J_{t}\left(gf\right)/g^{2}=\sum_{n=0}^{\infty}\log\left(\frac{1-q}{2}q^{(2n-1)/4}\right)c_{n}H_{n}/g.
\]
 The proof is will be complete after letting 
\[
h=f/g,f=gh.
\]
\end{proof}
\begin{cor}
For $0<q<1$ we have
\begin{equation}
\lim_{y\to\infty}\log q^{-y}\int_{0}^{\infty}\left(\frac{2q^{1/4}}{1-q}\right)^{a}q^{ay}T_{a}da=I\label{eq:5.10}
\end{equation}
and
\begin{equation}
\lim_{y\to\infty}\log q^{-y}\int_{0}^{\infty}\left(\frac{2q^{1/4}}{1-q}\right)^{a}q^{ay}S_{a}da=I\label{eq:5.11}
\end{equation}
 in strong operator topology.
\end{cor}

\begin{proof}
Since $\left\{ \frac{(1-q)^{a}}{2^{a}q^{a/4}}q^{ma/2},\ g(x)H_{m}(x\vert q)\right\} _{m=0}^{\infty}$
form a complete eigensystem for $T_{a}$, then the system $\left\{ \frac{(1-q)^{a}}{2^{a}q^{a/4}}q^{ma/2},\ H_{m}(x\vert q)/g(x)\right\} _{m=0}^{\infty}$
is  a complete eigensystem for $S_{a}$, hence $f\in L_2 [-1,1,w_{H}]$
implies 
\[
f(x)=g(x)\sum_{n=0}^{\infty}c_{n}H_{n}(x\vert q),\ f(x)=\sum_{n=0}^{\infty}f_{n}H_{n}(x\vert q)/g(x)
\]
and
\[
\sum_{n=0}^{\infty}|c_{n}|^{2}<\infty,\quad\sum_{n=0}^{\infty}|f_{n}|^{2}<\infty
\]
such that
\[
T_{a}f=\frac{(1-q)^{a}}{2^{a}q^{a/4}}g(x)\sum_{n=0}^{\infty}c_{n}q^{ma/2}H_{n}(x\vert q)
\]
and 
\[
S_{a}f=\frac{(1-q)^{a}}{2^{a}q^{a/4}}\sum_{n=0}^{\infty}f_{n}q^{ma/2}H_{n}(x\vert q)/g(x).
\]
 Then for $y>0$ we have
\[
\int_{0}^{\infty}\left(\frac{2q^{1/4}}{1-q}\right)^{a}q^{ay}\left(T_{a}f\right)(x)da=\frac{g(x)}{\log q^{-1}}\sum_{n=0}^{\infty}\frac{c_{n}}{y+m/2}H_{n}(x\vert q)
\]
and
\[
\int_{0}^{\infty}\left(\frac{2q^{1/4}}{1-q}\right)^{a}q^{ay}\left(S_{a}f\right)(x)da=\frac{1}{\log q^{-1}}\sum_{n=0}^{\infty}\frac{f_{n}}{y+m/2}H_{n}(x\vert q)/g(x).
\]
Then the corollary follows by taking limits. 
\end{proof}

 We next record the inversion formulas for both $T_a$ and $S_a$. 
 Let $\lfloor{a}\rfloor$ and $\{a\}$ denote the integer and fractional parts of $a$, respectively. 
 
 \begin{thm} \label{invTaSa}
	The left-inversion formulas of $T_a$ and $S_a$ are given by \\
	$\textup{(a)}$ $f(x) = (\D^{1+\lfloor{a}\rfloor}T_{1-\{a\}} g)(x)$
	if   $g(x) = (T_af)(x)$.\\
	$\textup{(b)}$ $f(x) = ( \mathcal{C}_q^{\lfloor{a}\rfloor+1}S_{1-\{a\}}g)(x)$
	if  $g(x) = (S_af)(x)$. 
\end{thm}
\begin{proof}
	It is clear that if $g(x) = (T_af)(x)$ then 
	\begin{eqnarray}
	\notag
	\bg
	\left[ \D^{\lfloor{a}\rfloor+1}T_{1-\{a\}+\epsilon}g(x)\right]
	=\left[  \D^{\lfloor{a}\rfloor+1}  T_{\lfloor{a}\rfloor+1+\epsilon} f(x)\right]  \\  
	=\D T_{1+\epsilon} f(x),
	\eg
	\end{eqnarray}
	holds for any $\epsilon >0$. 
	Now let $\epsilon \to 0^+$. The proof for $S_a$ is identical. 
\end{proof}

 \section{Application to the Askey--Wilson  polynomials} 
 We recall that the Askey--Wilson  polynomials are, \cite{Ismbook},  
 \begin{eqnarray}
 \bg
 p_n(\cos \t; a, b, c, d) = (ab, ac, ad; q)_na^{-n} \qquad \\
 \qquad \times 
 {}_4\f_3\left(\left. \ba{c} 
q^{-n}, q^{n-1} abcd,  ae^{i\t}, ae^{-i\t}
\\
ab, \quad ac, \quad ad
 \ea \right|q, q \right).
 \eg
 \end{eqnarray}
 \begin{thm}
 The operators $T_a$ have the property 
 \begin{eqnarray}
 \label{eq5phi4}
 \bg
T_a p_n(\cdot; -q^{1/4}, b, c, d) =  
 (-1)^n \; (-q^{1/4}b, - q^{1/4}c,-q^{1/4}d;q)_n
 \\
 \times  
 \frac{(1-q)^a (q^{a+1};q)_\infty}
{2^aq^{a/4+n/4}(q;q)_\infty}\;  \frac{(-q^{1/4} e^{i\t},
 -q^{1/4} e^{-i\t};q^{1/2})_\infty}{ (-q^{1/4+a/2} e^{i\t},
 -q^{1/4+a/2} e^{-i\t};q^{1/2})_\infty}  \\
 \qquad \times
 {}_5\f_4\left(\left. \ba{c}
q^{-n}, -q^{n-3/4} bcd, q,  -q^{a/2+1/4} e^{i\t}, -q^{a/2+1/4}e^{-i\t}
\\
-q^{1/4}b, \quad - q^{1/4}c, \quad -q^{1/4}d,  \quad q^{a+1}
 \ea \right|q, q \right).
 \eg
 \end{eqnarray}
 \end{thm}
 \begin{proof}
 It is clear that 
 \begin{eqnarray}
 \notag
 \bg
\frac{(-1)^nq^{n/4}}{(-q^{1/4} b, -q^{1/4} c, -q^{1/4} d;q)_n}   \frac{2^aq^{a/4}} {(1-q)^a (q^a;q)_\infty}\; T_ap_n(\cdot ; -q^{1/4}, b,c, d)   \\
=\sum_{k=0}^n \frac{(q^{-n}, -q^{n-3/4}bcd;q)_k}
{(q,-q^{1/4} b, -q^{1/4} c, -q^{1/4} d;q)_k} q^k  
\int_0^\pi \frac{(q, e^{2i\f},e^{-2i\f};q)_\infty }{h(\cos \f; -q^{k+1/4}, 
-q^{3/4}, q^{a/2} e^{i\t}, q^{a/2} e^{-i\t})}    \frac{ d\f}{2\pi}. 
 \eg
 \end{eqnarray}
 The integral is an Askey--Wilson  integral  and the above expression 
 simplifies to establish 
 \end{proof}
 The special case when $b = -q^{3/4}$ is very interesting. In this case we have 
 \begin{eqnarray}
 \bg
 T_a p_n(\cdot; -q^{1/4},   -q^{3/4}, c, d) =  
 (-1)^n \; (q, - q^{1/4}c,-q^{1/4}d;q)_n   
 \\
 \times   
 \frac{(1-q)^a( q^{a+1};q)_\infty}
{2^aq^{a/4+n/4}(q;q)_\infty}\;  \frac{(-q^{1/4} e^{i\t}, 
 -q^{1/4} e^{-i\t};q^{1/2})_\infty}{ (-q^{1/4+a/2} e^{i\t}, 
 -q^{1/4+a/2} e^{-i\t};q^{1/2})_\infty}  \\
 \qquad \times 
 {}_4\f_3\left(\left. \ba{c} 
q^{-n}, q^{n} cd,  -q^{a/2+1/4} e^{i\t}, -q^{a/2+1/4}e^{-i\t}
\\
-q^{1/4}c, \quad -q^{1/4}d,  \quad q^{a+1}
 \ea \right|q, q \right).
 \eg
 \end{eqnarray}
 The ${}_4\f_3$ is  a multiple of an Askey--Wilson  polynomial. Indeed we proved the connection relation 
  \begin{eqnarray}
  \label{eqConnAWP}
 \bg
 T_a p_n(\cdot; -q^{1/4},   -q^{3/4}, c, d) =  \frac{(q;q)_n}{(q^{a+1};q)_n} q^{an/2} \frac{(1-q)^a(q^{a+1};q)_\infty}
{2^aq^{a/4}(q;q)_\infty}\; 
 \\
 \times   
  \frac{(-q^{1/4} e^{i\t}, 
 -q^{1/4} e^{-i\t};q^{1/2})_\infty}{ (-q^{1/4+a/2} e^{i\t}, 
 -q^{1/4+a/2} e^{-i\t};q^{1/2})_\infty}  
\; p_n(x; -q^{a/2+1/4}, - q^{3/4+a/2}, cq^{-a/2}, dq^{-a/2}).
 \eg
 \end{eqnarray}
  
  Our next task is to derive a bilinear formula which follows from the Hilbert-Schmidt decomposition of a symmetric kernel of an integral operator.  We follow the technique develeoped in \cite{Ism}. 
  
  We recall the orthogonality relation of the  Askey-Wilson polynomial  \cite{And:Ask:Roy}, \cite{Gas:Rah}, \cite{Ismbook}    
\begin{eqnarray}\label{eqawpo}
\bg
\int_{0}^{\pi}p_m(\cos\t;\textbf{t}|q)p_n(\cos\t;\textbf{t}|q)w(\cos\t;\textbf{t})d\t\\
=\frac{2\pi(t_1t_2t_3t_4q^{2n};q)_{\infty}(t_1t_2t_3t_4q^{n-1};q)_{n}}{(q^{n+1};q)_{\infty}\prod_{1\le j<k\le 4}(t_jt_kq^{n};q)_{\infty}}\delta_{m,n}
\eg
\end{eqnarray}
 where ${\bf t}= (t_1, t_2, t_3, t_4)$, 
 \begin{eqnarray}
 w(\cos\t;t_1,t_2,t_3,t_4): =\frac{(e^{2i\t},e^{-2i\t};q)_{\infty}}{\prod_{j=1}^4(t_je^{i\t},t_je^{-i\t};q)_{\infty}}.
 \end{eqnarray}
 Then we define three sequences $\{A_n\}$, $\{B_n\}$ and $\{C_n\}$. Let 
 \begin{eqnarray}
 \notag
 \bg
 M_n(t_1,t_2,t_3,t_4)  =\frac{2\pi(t_1t_2t_3t_4q^{2n};q)_{\infty}(t_1t_2t_3t_4q^{n-1};q)_{n}}{(q^{n+1};q)_{\infty}\prod_{1\le j<k\le 4}(t_jt_kq^{n};q)_{\infty}},\\
A_n=M_n(-q^{a/2+1/4},-q^{a/2+3/4},q^{-a/2}c,q^{-a/2}d),\\
 B_n=M_n(-q^{1/4},-q^{3/4},c,d),\qquad 
C_n=\frac{(q;q)_n}{(q^a;q)_{n+1}(q;q)_\infty}q^{an/2}.
 \eg
 \end{eqnarray}

 Note that  \eqref{eqConnAWP} shows that  $T_a$ maps an Askey-Wilson polynomial to another Askey-Wilson polynomial and \eqref{eqawpo} show its orthogonality relation, hence we drive the following the integral equation,
 \begin{eqnarray}\label{int1}
 \bg
 \int_{0}^{\pi}w_0(\cos\t|q )\int_0^\pi  
 \frac{ w_H(\cos \phi_1|q) p_n(\cos\f_1|-q^{1/4},-q^{3/4},c,d)}
 	{(-q^{1/4} e^{i\f_1}, -q^{1/4} e^{-i\f_1};q^{1/2})_\infty  h(\cos \f_1; q^{a/2}e^{i\t}, q^{a/2} e^{-i\t})} \sin \f_1 \\
 	\int_0^\pi  
 	\frac{ w_H(\cos \f_2|q) p_m(\cos\f_2|-q^{1/4},-q^{3/4},c,d)}
 	{(-q^{1/4} e^{i\f_2}, -q^{1/4} e^{-i\f_2};q^{1/2})_\infty  h(\cos \f_2; q^{a/2}e^{i\t}, q^{a/2} e^{-i\t})} \sin \f_2  d\f_1 d\f_2 d\t\\
 	=A_nC_n^2 \delta_{m,n},
 	\eg
 	\end{eqnarray}
 where 
 \begin{eqnarray}
 \notag
 w_0(\cos\t|q)=w(\cos\t;-q^{a/2+1/4},-q^{a/2+3/4},q^{-a/2}c,q^{-a/2}d|q)\; (-q^{1/4+a/2}e^{i\t},-q^{1/4+a/2}e^{-i\t};q^{1/2})^2_{\infty}.
 \end{eqnarray}
  If we change the order of integration, which leads to 
  \begin{eqnarray}
  \label{int2}
  \bg
  \int_{0}^{\pi}\int_{0}^{\pi}\frac{ w_H(\cos \phi_1|q) p_n(\cos\f_1|-q^{1/4},-q^{3/4},c,d)\sin \f_1}
  {(-q^{1/4} e^{i\f_1}, -q^{1/4} e^{-i\f_1};q^{1/2})_\infty }\\ \times
  \frac{ w_H(\cos \f_2|q) p_m(\cos\f_2|-q^{1/4},-q^{3/4},c,d) \sin \f_2}
  {(-q^{1/4} e^{i\f_2}, -q^{1/4} e^{-i\f_2};q^{1/2})_\infty } \ \\
  \int_{0}^{\pi} \frac{w_0(\cos\t|q )\;\;  d\t\;  d\f_1\; d\f_2}{h(\cos \f_1; q^{a/2}e^{i\t}, q^{a/2} e^{-i\t})h(\cos \f_2; q^{a/2}e^{i\t}, q^{a/2} e^{-i\t})} 	=A_nC_n^2 \delta_{m,n}.
  \eg
  \end{eqnarray}

 The polynomials  
 $p_n(x|-q^{a/2+1/4},-q^{a/2+3/4},q^{-a/2}c,q^{a/2}d)$ are uniquely determined by orthogonality relation \eqref{eqConnAWP}.
 From \eqref{int1} and uniqueness, we obtain kernel function $K(\cos\f_1,\cos\f_2)$ and it has the following connection relation:
\begin{equation}\label{kernel}
	 \frac{A_nC_n^2}{B_n} p_n(\cos\f_2|-q^{1/4},-q^{3/4},c,d)=\int_{0}^{\pi}K(\cos\f_1,\cos\f_2)
	p_n(\cos\f_1|-q^{1/4},-q^{3/4},c,d)d\f_1,
\end{equation}
where the Kernel function $K(\cos\f_1,\cos\f_2)$ is 
\begin{eqnarray}
\bg
K(\cos\f_1,\cos\f_2)= \frac{w_H(\cos \f_1|q)}{(-q^{1/4} e^{i\f_1}, -q^{1/4} e^{-i\f_1};q^{1/2})_\infty } \; \frac{w_H(\cos \f_2|q)}{(-q^{1/4} e^{i\f_2}, -q^{1/4} e^{-i\f_2};q^{1/2})_\infty }\\
\times\frac{\sin\f_1\;  \sin\f_2}{w(\cos\f_2;-q^{1/4},-q^{3/4},c,d)}\\
\times 
\int_{0}^{\pi}\frac{w_0(\cos\t|q)}{h(\cos \f_1; q^{a/2}e^{i\t}, q^{a/2} e^{-i\t})h(\cos \f_2; q^{a/2}e^{i\t}, q^{a/2} e^{-i\t})}d\t.
\eg
\end{eqnarray}
The connection relation \eqref{kernel} implies the bilinear formula as follows
\begin{eqnarray}
\label{kernel2}
\bg
\frac{K(\cos\f_1,\cos\f_2)}{w(\cos\f_1|-q^{1/4},-q^{3/4},c,d)}= \sum_{n=0}^{\infty} A_n\left(\frac{C_n}{B_n}\right)^2p_n(\cos\f_1|-q^{1/4},-q^{3/4},c,d)\\\times p_n(\cos\f_2|-q^{1/4},-q^{3/4},c,d).
\eg
\end{eqnarray}
 
  Clearly, the left-side of \eqref{kernel2} is continuous, square integral and symmetric kernel.

 \section{Approximation Operators}
 We let 
 \begin{eqnarray}
 \label{eqdefej}
 e_j(x) = x^j.
 \end{eqnarray}
 Our first result is the following expansion 
  \begin{eqnarray}
  \label{eqexinHn}
 \bg
 \frac{1}{(-q^{1/4}e^{i\t}, -q^{1/4}e^{-i\t};q^{1/2})_\infty}
 = \Sum \frac{(-1)^n q^{n/2}}{(q;q)_n(q;q)_\infty} H_n(\z|q)H_n(\cos \t|q), 
 \eg
 \end{eqnarray}
 where $\z = (q^{1/4}+q^{-1/4})/2$. 
\begin{proof}[Proof of \eqref{eqexinHn}]
Let the left-hand side of \eqref{eqexinHn} be 
$\Sum c_n H_n(x|q)$. Then 
\begin{eqnarray}
\notag
c_n(q;q)_n = \int_0^\pi  \frac{w_H(\cos \t|q)H_n(\cos\t|q) }{(-q^{1/4}e^{i\t}, -q^{1/4}e^{-i\t};q^{1/2})_\infty}  \sin \t d\t.
\end{eqnarray}
In view of the generating function \eqref{eqGFHn} we see 
that 
\begin{eqnarray}
\notag
\Sum c_nt^n =  \int_0^\pi \frac{w_H(\cos \t|q)\sin \t \; d\t }{(-q^{1/4}e^{i\t}, -q^{1/4}e^{-i\t};q^{1/2})_\infty(te^{i\t}, te^{-i\t};q)_\infty}.
\end{eqnarray}
 The  above integral  is a special Askey--Wilson integral. Indeed   
 \begin{eqnarray}
 \notag
 \bg
   \int_0^\pi \frac{w_H(\cos \t|q)\sin \t \; d\t }{(-q^{1/4}e^{i\t}, -q^{1/4}e^{-i\t};q^{1/2})_\infty(te^{i\t}, te^{-i\t};q)_\infty} \\
 = \frac{(q;q)_\infty}{2\pi} \int_0^\pi \frac{(e^{2i\t}, e^{-2i\t};q)_\infty d\t }{(-q^{1/4}e^{i\t}, -q^{1/4}e^{-i\t},-q^{3/4}e^{i\t}, -q^{3/4}e^{-i\t}, te^{i\t}, te^{-i\t};q)_\infty}\\
 = \frac{1}{(q, -q^{1/4}t, -q^{3/4}t;q)_\infty} = \frac{1}{
 (q;q)_\infty}\Sum \frac{(-t\sqrt{q})^n}{(q;q)_n} H_n(\z|q),
 \eg
 \end{eqnarray}
 where $\z = (q^{1/4}+q^{-1/4})/2$. Therefore 
 \begin{eqnarray}
 \notag
\int_0^\pi  \frac{w_H(\cos \t|q)  H_n(\cos \t|q)\, \sin \t \; d\t }{(-q^{1/4}e^{i\t}, -q^{1/4}e^{-i\t};q^{1/2})_\infty} =(-1)^n
\frac{q^{n/2}}{(q;q)_\infty} H_n(\z|q),
 \end{eqnarray}
 which implies \eqref{eqexinHn} and the proof is complete. 
 \end{proof}
 
 Therefore the action of $T_a$ on the constant function 1 is 
 \begin{eqnarray}
 \notag
 \bg
T_a e_0(x) = \frac{(1-q)^a}{2^aq^{a/4}}
(-q^{1/4}e^{i\t}, -q^{1/4}e^{-i\t};q^{1/2})_\infty
\int_0^\pi  \Sum  \frac{q^{n/2}(-1)^n}{(q;q)_n(q;q)_\infty} H_n(\z|q)H_n(\cos \f|q) \\
\times
\sum_{m=0}^\infty H_m(\cos \t|q) H_m(\cos \f|q)\frac{q^{am/2}}{(q;q)_m} w_H(\cos \f)\; \sin \f\; d\f\\
= \frac{(1-q)^a}{2^aq^{a/4}(q;q)_\infty}  
(-q^{1/4}e^{i\t}, -q^{1/4}e^{-i\t};q^{1/2})_\infty
\Sum  \frac{(-1)^n q^{n(a+1)/2}}{(q;q)_n} H_n(\cos \t|q) 
H_n(\z|q). 
 \eg
 \end{eqnarray}
 Therefore we have proved that 
 \begin{eqnarray}
 \label{eqTae0}
 \bg
 T_a e_0(\cos\t) =  \frac{(1-q)^a(q^{a+1};q)_\infty}{2^aq^{a/4}(q;q)_\infty}  \qquad   \qquad \qquad  \qquad \qquad\\
 \times 
\frac{(-q^{1/4}e^{i\t}, -q^{1/4}e^{-i\t};q^{1/2})_\infty}
{(-q^{(a+1/2)/2} e^{i\t}, - q^{(a+1/2)/2} e^{-i\t}, 
-q^{(a+3/2)/2} e^{i\t}, - q^{(a+3/2)/2} e^{-i\t};q)_\infty}.
 \eg
 \end{eqnarray}
 Carlitz \cite{Car} proved a bilinear generating function which is equivalent to 
 \begin{eqnarray}
 \label{eqCarlitz}
 \bg
  \Sum H_n(\cos \t|q) H_{n+m}(\cos\f|q) \frac{t^n}{(q;q)_n} \\
  = \frac{(t^2;q)_\infty}{(te^{i(\t+\f)}, te^{i(\t- \f)}, te^{i(\f+\t)}, 
  te^{-i(\t+\f)};q)_\infty}
   \sum_{j=0}^m \gauss{m}{j} \frac{(te^{i(\t+\f)}, 
  te^{i(\f-\t};q)_j}{(t^2;q)_j}\; e^{i\f(m-2j)}. 
  \eg
 \end{eqnarray}
 The special case $m=1$ is 
 \begin{eqnarray}
 \label{eqLemma7.1}
 \bg
 \Sum H_n(\cos \t|q) H_{n+1}(\cos \f|q) \frac{t^n}{(q;q)_n}  \qquad \qquad\\  \qquad \qquad
 =   \frac{2(\cos \f -t\cos \t)\; (qt^2;q)_\infty}
 {(t e^{i(\t+\f)}, t e^{i(\t-\f)}, t e^{-i(\t+\f)}, 
 t e^{i(\f-\t)};q)_\infty}.
 \eg
 \end{eqnarray}
 and can be directly proved by 
 applying  $\D$ to the Poisson kernel \eqref{eqPoiKer} and making use of the lowering relation 
 \begin{eqnarray}
 \label{eqDHn}
 \D H_n(x|q) = 2 \frac{(1-q^n)}{1-q} q^{(1-n)/2} H_{n-1}(x|q), 
 \end{eqnarray}
 \cite[(13.1.21)]{Ismbook}. Ismail and Simeonov  \cite{Ism:Sim} 
discuss an operational approach to the approach  to 
\eqref{eqLemma7.1}.  

Ismail and Stanton \cite{Ism:Sta2} gave an alternate representation of the right-hand side of \eqref{eqCarlitz}. Their representation is 
\begin{eqnarray}
 \label{eqCarlitzIsmSta}
 \bg
  \Sum H_n(\cos \t|q) H_{n+m}(\cos\f|q) \frac{t^n}{(q;q)_n} 
  = \frac{(t^2q^m;q)_\infty}{(te^{i(\t+\f)}, te^{i(\t- \f)}, te^{i(\f+\t)}, 
  te^{-i(\t+\f)};q)_\infty} \\
  \times 
   \sum_{k=0}^m \gauss{m}{k} (te^{i(\f-\t)};q)_k 
  (te^{i(\t-\f)};q)_{m-k} e^{i(m-2k)\f}. 
  \eg
 \end{eqnarray}
 This representation shows that the $k$ sum is polynomial in 
 $t$ of degree $m$ while the finite sum in Carlitz's formula 
 is a polynomial of degree at most $2m$. There is actually 
 a huge amount of cancelations in Carlitz's formula. 

\begin{thm}
The action of $T_a$ on $e_0$ and $e_1$ is given by  
 \begin{eqnarray}
 \label{eqTae02}
 T_a e_0(\cos\t) = \frac{(1-q)^a(q^{a+1};q)_\infty}{2^aq^{a/4}(q;q)_\infty}  \frac{(-q^{1/4}e^{i\t}, -q^{1/4}e^{-i\t};q^{1/2})_\infty}
{(-q^{(a+1/2)/2} e^{i\t}, - q^{(a+1/2)/2} e^{-i\t};q^{1/2})_\infty}
\end{eqnarray}
and 
\begin{eqnarray}
\label{eqTae1}
\bg
 T_a e_1(x) =  \frac{(1-q)^a(q^{a+2};q)_\infty}{2^aq^{a/4}(q;q)_\infty}[xq^{a/2}(1-q)+\z q^{1/2}
(q^a-1)] 
 \\ \qquad \times  \frac{(-q^{1/4}e^{i\t}, -q^{1/4}e^{-i\t};q^{1/2})_\infty}
{(-q^{(a+1/2)/2} e^{i\t}, - q^{(a+1/2)/2} e^{-i\t};q^{1/2})_\infty},
\eg
\end{eqnarray}
respectively. 
\end{thm} 
\begin{proof}
Equation \eqref{eqTae02} is a restatement of \eqref{eqTae0}. 
It is clear from the three term recurrence rekation in \eqref{eq2.1} that 
\begin{eqnarray}
\notag
\bg
2  \frac{2^aq^{a/4}} {(1-q)^a} \frac{T_a e_1(\cos\t)}
{(-q^{1/4}e^{i\t}, -q^{1/4}e^{-i\t};q^{1/2})_\infty}  \\
  = \int_0^\pi 2 \cos \f \Sum \frac{(-1)^n q^{n/2}}{(q;q)_n(q;q)_\infty} H_n(\z|q)H_n(\cos \f|q) w_H(\cos \f|q)\\
  \times  \sum_{m=0}^\infty 
  H_m(\cos \t|q) \frac{H_m(\cos \f|q)}{(q;q)_m} q^{ma/2} \sin \f 
 \; d\f\\
 = \sum_{m,n=0}^\infty  q^{ma/2} H_m(\cos \t|q) \frac{(-1)^n q^{n/2} H_n(\z|q)}{(q;q)_n(q;q)_m(q;q)_\infty}\\
 \times \int_0^\pi   \left(H_{n+1}(\cos \f|q) + (1-q^n) 
 H_{n-1}(\cos \f|q)\right) H_m(\cos \f|q) w_H(\cos \f|q) \sin \f 
 \; d\f.
\eg
\end{eqnarray}
We now use the orthogonality relation to simplify the last line 
to 
\begin{eqnarray}
(q;q)_m[\delta_{m,n+1} +  (1-q^n) \delta_{m,n-1}].
\notag
\end{eqnarray}
We apply \eqref{eqLemma7.1} and conclude that
\begin{eqnarray}
\notag
\bg
2  \frac{2^aq^{a/4}} {(1-q)^a}\ T_a e_1(\cos\t)  = 
\frac{2(q^{a+2};q)_\infty}{(q;q)_\infty}[xq^{a/2}(1-q)+\z q^{1/2}
(q^a-1)] \\
  \qquad \times  \frac{(-q^{1/4}e^{i\t}, -q^{1/4}e^{-i\t};q^{1/2})_\infty}
{(-q^{(a+1/2)/2} e^{i\t}, - q^{(a+1/2)/2} e^{-i\t};q^{1/2})_\infty}.
\eg
 \end{eqnarray}
 \end{proof}
 
 The next step is to find the rate of approximation of a continuous function by $T_a$.  To carry this out we will need the following lemmas.
 \begin{lem}
 The action of $T_a$ on $e_2$ is given by 
 \begin{eqnarray}
 \label{eqTaone2}
 \bg
 (T_ae_2)(x) =\frac{(1-q)^a(q^{a+3};q)_\infty}{2^aq^{a/4}(q;q)_{\infty}} \frac{(-q^{1/4}e^{i\t}, -q^{1/4}e^{-i\t};q^{1/2})_\infty}{(-q^{(a+1/2)/2} e^{i\t},-q^{(a+1/2)/2} e^{-i\t};q^{1/2})_\infty}\\
\{q^a(1-q)(1-q^2)x^2+\frac{1}{2}[(q^a+q)(q^{a/2+1/4}+q^{a/2+3/4}+q^{a/2+5/4}+q^{a/2+7/4})\\-(1+q)(q^{a/2+1/4}+q^{a/2+3/4}+q^{3a/2+5/4}+q^{3a/2+7/4})]x\\+\frac{1}{4}[(1+q)(q^{1/2}+q+q^a+q^{a+2}+q^{2a+1}+q^{2a+3/2}-1-q^{a+1/2}-q^{a+3/2}\\-q^{2a+2})+2(1-q^{a+1})(1-q^{a+2})-2(q^a+q^{a+3})]\} \eg
 \end{eqnarray}
 \end{lem}
 \begin{proof}
 Note that \eqref{eqqHrr} implies 
 \begin{eqnarray}
 \notag
4x^2 H_n(x|q) = H_{n+2}(x|q) +(2-q^n -q^{n+1}) H_n(x|q) 
+ ( 1-q^n)(1 -q^{n-1}) H_{n-2}(x|q),
\end{eqnarray}
 hence, using \eqref{eqCarlitzIsmSta} for special cases $m=0$ and $m=2$, we obtain
 \begin{eqnarray}
 \notag
 \bg 
4 (q;q)_\infty  \frac{2^a q^{a/4}} {(1-q)^a} 
\frac{T_a e_2(\cos\t)}  {(-q^{1/4}e^{i\t}, -q^{1/4}e^{-i\t};q^{1/2})_\infty}  
\\= 
q^a \Sum (-1)^n \frac{q^{n(a+1)/2}}{(q;q)_n}H_n(\z|q)H_{n+2}(x|q) \\
+ \Sum (-1)^n \frac{q^{n(a+1)/2}}{(q;q)_n}
[2-q^n-q^{n+1}]H_n(\z|q)H_{n}(x|q) \\
+ q \Sum (-1)^n \frac{q^{n(a+1)/2}}{(q;q)_n}H_n(x|q)H_{n+2}(\z|q).
\eg
\end{eqnarray}
Thus the above expression is 
\begin{eqnarray}
\notag
\bg
= \frac{2(q^{a+1};q)_{\infty}}{(-q^{a/2+1/4}e^{i\t},-q^{a/2+1/4}e^{-i\t};q^{1/2})_{\infty}}\\ - \frac{(1+q)(q^{a+3};q)_{\infty}}{(-q^{a/2+5/4}e^{i\t},-q^{a/2+5/4}e^{-i\t};q^{1/2})_{\infty}}\\
+\frac{(q^{a+3}:q)_{\infty}}{(-q^{a/2+1/4}e^{i\t},-q^{a/2+1/4}e^{-i\t};q^{1/2})_{\infty}}[2(q^a+q^{a+3})\cos 2\t+\\2(q^a+q)(q^{a/2+1/4}+q^{a/2+3/4}+q^{a/2+5/4}+q^{a/2+7/4})\cos \t\\
+(1+q)(q^{1/2}+q+q^a+q^{a+2}+q^{2a+1}+q^{2a+3/2})].
 \eg
 \end{eqnarray}
 \end{proof}
 \begin{lem}\label{lem7.3}
 We have 
 \begin{eqnarray}
 T_a((x-y)^2) = \O(a),
 \end{eqnarray}
 as $a \to 0$ uniformly in a neighborhood of $a = 0$. 
 \end{lem}
 \begin{proof}
	Let $x=\cos\t$ and $y=\cos \phi$, using \eqref{eqTae02}, \eqref{eqTae1} and \eqref{eqTaone2}, we derive the following formula
	\begin{eqnarray}
	\bg
	x^2T_a e_0-2xT_a e_1+T_a e_2=\frac{(1-q)^a(q^{a+3};q)_\infty}{2^aq^{a/4}(q;q)_{\infty}} \frac{(-q^{1/4}e^{i\t}, -q^{1/4}e^{-i\t};q^{1/2})_\infty}{(-q^{(a+1/2)/2} e^{i\t},-q^{(a+1/2)/2} e^{-i\t};q^{1/2})_\infty}\\
	\times \{[(1-q^{a+1})(1-q^{a+2})-2q^{a/2}(1-q^{a+2})(1-q)+q^a(1-q)(1-q^2)]x^2\\+[\frac{1}{2}(q^a+q)(q^{a/2+1/4}+q^{a/2+3/4}+q^{a/2+5/4}+q^{a/2+7/4})\\-\frac{1}{2}(1+q)(q^{a/2+1/4}+q^{a/2+3/4}+q^{3a/2+5/4}+q^{3a/2+7/4})\\-(1-q^{a+2})(q^{3/4}+q^{1/4})(q^a-1)]x+\\   
\frac{1}{4}[(1+q)(q^{1/2}+q+q^a+q^{a+2}+q^{2a+1}+q^{2a+3/2}-1\\ 
-q^{a+1/2}-q^{a+3/2}-q^{2a+2})+2(1-q^{a+1})(1-q^{a+2})-2(q^a+q^{a+3})]\}.
\eg
\end{eqnarray}
In fact, the quantity in the square bracket is 
\begin{eqnarray}
\notag
\bg
{}[(1-q^{a+1})(1-q^{a+2})-2q^{a/2}(1-q^{a+2})(1-q)+q^a(1-q)(1-q^2)]x^2\\+[\frac{1}{2}q^{a/2+1/4}(1+q^{1/2})(1+q)(q^a+q-1-q^{a+1})\\-(1-q^{a+2})(q^{3/4}+q^{1/4})(q^a-1)]x\\+\frac{1}{4}[(1+q)((q^{1/2}+q)(1+q^{2a+1/2})-(1+q^{a+1/2})(1+q^{a+3/2})-q^{a}(1+q^{2}))\\+2(1+q^{2a+3})]
	=x^2(q^2-q)a\log q -\frac{ax}{2}(1-q)(1+q)(q^{1/4}+q^{3/4})\log q\\+\frac{1}{4}[(1+q)(-1-q^{1/2}+2q+q^{3/2}+q^2)-4q^2]a\log q +\O(a^2).
	\eg
	\end{eqnarray}
\end{proof}
 
 The next theorem gives the order of the error term in the approximation of a function by $T_a$. 
 
 \begin{thm}\label{thmOrderofApp} Assume that $f$ is twice continuously differentiable. Then 
 \begin{eqnarray}
 \|T_a f - f\| = \O(a).
 \end{eqnarray}
 \end{thm}
  \begin{proof}
 We write $x = \cos \t, y= \cos \f$. Write $f(y) = f(x) + (y-x)f^\prime(x) + (y-x)^2 f^{\prime\prime}(u)$. Using \eqref{eqTae02}--\eqref{eqTae1} we see that for fixed $x$
 \begin{eqnarray}
 \bg
T_a[f(x) + (y-x)f^\prime(x)] = [f(x) -xf^\prime(x)]T_ae_0 
+ f^\prime(x) T_ae_1 \\
= \frac{(1-q)^a(q^{a+2};q)_\infty}{2^aq^{a/4}(q;q)_\infty}  \frac{(-q^{1/4}e^{i\t}, -q^{1/4}e^{-i\t};q^{1/2})_\infty}
{(-q^{(a+1/2)/2} e^{i\t}, - q^{(a+1/2)/2} e^{-i\t};q^{1/2})_\infty}\\
\times \left[f(x) (1-q^{a+1}) + (1-q^{a/2})f^\prime(x)
 \{x(q^{1+a/2}-1)-\z q^{1/2}(1+q^{a/2})\}\right]. 
 \notag
 \eg
 \end{eqnarray}
 The quantity in the square bracket is clearly $\O(a)$. Assume that 
 $\| f^{\prime\prime}\| = M$. Then 
 \begin{eqnarray}
 \notag
 \|(T_af)(x) -f(x)\| \le  \O(a) + M \|T_a(e_2) -2x T_ae_1 + e_2(x)T_a(e_0)\| = \O(a),
\end{eqnarray}
by Lemma \ref{lem7.3}. 
\end{proof}

  \section{$T_a$ are Contraction Maps}
  In this section we prove that the operators $T_a$ are contractions maps. 
  
  \begin{thm}
 For  any $q$ satisfies 
  \begin{eqnarray}
  \label{eqcond}
 \frac{(1-q)}{2q^{1/4}} < 1,
  \end{eqnarray}
 there exists a positive number $c(q)>0$ such that for all $a>c(q)$  the operators  $T_a$  are contraction operators on $C[-1, 1]$. The condition \eqref{eqcond} is also necessary if $T_a$ is a contraction operator for any  $a>0$.
   \end{thm}
  \begin{proof} It is clear from \eqref{eqeigenvalues} and the fact that the norm is $\ge \l_0$ that the condition  \eqref{eqcond} is necessary.  It remain to show that \eqref{eqcond} is sufficient. 
 Note that $T_a$ are positive linear operators, hence we only need to show that $\|T_ae_0\| < 1$, that is $|(T_ae_0)(x)| < 1$ for all $x \in [-1,1]$ and all $ a $ large enough.  Observe that
  \begin{eqnarray}
  \notag
  \frac{1+ 2xt + t^2}{1+ 2xu + u^2}
  \end{eqnarray} 
 increases with $x$ for $x \in [-1,1], 0 \le u < t \le 1$. In view of 
 \eqref{eqTae0}, the above observation reduces the problem, $\|T_ae_0\| < 1$, to showing that 
 \begin{eqnarray}
  \frac{(1-q)^a(q^{a+1};q)_\infty}{2^aq^{a/4}(q;q)_\infty}  
  \frac{(-q^{1/4};q^{1/2})_\infty^2}
{(-q^{(a+1/2)/2};q^{1/2})_\infty^2} < 1,
\notag
\end{eqnarray}
 for $a >c(q)$, for some $c(q) >0$.  We set 
 \begin{eqnarray}
 \label{eqdefh}
  \frac{(1-q)^a(q^{a+1};q)_\infty}{2^aq^{a/4}(q;q)_\infty}  
  \frac{(-q^{1/4};q^{1/2})_\infty^2}
{(-q^{(a+1/2)/2};q^{1/2})_\infty^2} = e^{h(a)}.
\end{eqnarray}
 It is clear that $h(0) =0, h(a) \to -\infty$ as $a \to +\infty$, and 
 \begin{eqnarray}
 h(a) = a\log \left(\frac{1-q}{2q^{1/4}}\right) + \Sum \log(1-q^{a+n+1}) 
 - 2\Sum \log \left(\frac{1+q^{(n+a+1/2)/2}}{1+q^{(n+1/2)/2}}\right) . 
 \notag
 \end{eqnarray}
Now 
\begin{eqnarray}
\notag
h^\prime(a) = \log \left(\frac{1-q}{2q^{1/4}}\right)  - \log q
\Sum \frac{q^{a+n+1}}{1- q^{a+n+1}} -
  \log q  \Sum \frac{q^{(a+n+1/2)/2}}{1+q^{(a+n+1/2)/2}}. 
\end{eqnarray}
Therefore $h^{\prime\prime}(a) < 0$, so that $h^\prime(a)$ is decreasing.  But $h^\prime(a) <0$ for sufficiently large $a$ and 
$\lim_{a \to -1}h^\prime(a) =+\infty$. Hence there is a $c >0$ such 
that  $h^\prime(a) <0$ for all $a >c$. Therefore $h$ strictly  decreasing and the results follows.
 \end{proof}
 
\section{Another Semigroup}
 In this section we study the semigroup 
 \begin{eqnarray}
 \label{eqdeFa}
 \bg
 (F_a f)(\cos \t) \\
 = 
  \int_0^\pi  
 \frac{ (q^a;q)_\infty\; w_H(\cos \f|q)\;  f(\cos \f) \;  \sin \f \; d\f}
 { (q^{a/2}e^{i(\t+\f)}, q^{a/2}e^{i(\t-\f)},  q^{a/2} e^{-i(\t+\f)}, 
 q^{a/2}e^{i(\f-\t)};q)_\infty}. 
 \eg
 \end{eqnarray}
 It is clear from the definition and the orthogonality of the $q$-Hermite polynomial that 
  \begin{eqnarray}
 \label{eqFaHn}
 F_a H_n(x|q) = q^{an/2} H_n(x|q). 
 \end{eqnarray} 
 It must be noted that if we took \eqref{eqFaHn} as our starting point, 
 that is we define $F_a$ as a multiplier operator on $L_2[-1,1, w_H]$ 
then computed the kernel of $F_a$ as an integral operator we 
would have naturally found the definition \eqref{eqdeFa}. This is 
similar to the way Weyl originally defined the Weyl fractional integrals, see \cite{Zyg}. 
 
 The infinitesimal generator can be defined in a way similar to the way we defined the infinitesimal  generators for $T_a$ and $S_a$.

 We now  consider $F_a$ as approximation operators. 
\begin{thm}\label{thm10.1}
The operators $F_a$ have the properties 
\begin{eqnarray}
\notag
\textup{(a)}\;   (F_a e_0)(x) = 1,\qquad 
\textup{(b)} \; (F_ae_1)(x) = q^{a/2} x ,   \qquad 
\textup{(c)} \;  (F_ae_2)(x)= q^a x^2 + \frac{(1-q)}{4} \; (1-q^a). 
\end{eqnarray}
\end{thm}
It is clear from \eqref{eqFaHn} that $(F_a H_m(\cdot\vert q))(x) \to H_m(x|q)$,
 uniformly on $[-1,1]$ as $ a \to 0$, for $m=0, 1, 2$. Moreover it is clear $T_a$ is self-adjoint on $L_2[-1,1, w_H]$.

 \begin{thm}\label{thm10.2} 
 The following statements are true\\
$\textup{(a)}$  The operators $F_a, a \ge 0$ form a semigroup of positive linear operators defined on $C[-1,1]$. \\
$\textup{(b)}$  $\lim_{a \to 0} (F_af)(x) = f(x)$,  for $f \in C[-1,1]$. \\
$\textup{(c)}$  $\|F_a\| =1, a \ge 0$, where the norm is for $F_a$ acting on $C[-1,1]$. \\
$\textup{(d)}$ The only eigenvalues of $F_a$, $a>0$, are $q^{ma/2}, 
m= 0, 1, \cdots$ with $H_n$ are the corresponding eigenfunctions. \\
$\textup{(e)}$ For every $a >0$, $F_a$ is a compact operator, hence 
is not invertible. 
 \end{thm}
 \begin{proof}
 The completeness of the $q$-Hermite polynomials  and the fact 
 that the kernel is the Poisson kernel establish the semigroup 
 property (a). In view of \ref{thm10.1}, 
 we can apply Korovkin's theorem to establish the $F_a$ tends to the identity operator as $ a \to 0+$, so (b) holds.  Part (c) follows from the positivity of the kernel of $F_a$, and the fact that $(F_a e_0)(x) = 1$. The proof of parts $(d)-(e)$ are similar to the case of $T_a$ and are omitted but they use 
 \eqref{eqFaHn}. 
 \end{proof}
 
These operators appeared in Suslov's work \cite{Sus} where he stated the   
semigroup property and $F_a \to I$ but he assumed that the functions   
are analytic in the unit disc. Parts (c)-(e) were not mentioned by Suslov.

 \begin{thm}\label{thm9.3}
 If $ a> 1$, then the operators $F_a$ satisfy the commutation relation 
 \begin{eqnarray}
 \label{eq10.3}
 \D F_a = q^{a/2}F_a \D
 \end{eqnarray}
 with 
 \begin{eqnarray}
 \label{eq10.4}
F_a \D f =  \frac{4}{(1-q)(1-q^{a-1})} \left(F_{a-1} \circ M -q^{(a-1)/2} M \circ F_{a-1}\right)f, 
\end{eqnarray}
where $(Mf)(x)=x \cdot f(x)$. Moreover,  $(\D F_af)(x) \ge 0$ if $\D f \ge 0$. 
 \end{thm}
  The proof uses the raising relation \cite{Ismbook} 
 \bea 
 \label{eqraingHn}
\D[ w_x(x|q) H_n(x|q)] = - \frac{2q^{-n/2}}{1-q} w_x(x|q) 
H_{n+1}(x|q)
 \eea
 \begin{proof}[Proof of Theorem \ref{thm9.3}]
 It is clear that 
 \bea
 \notag
 \bg
 (\D F_af)(x) = \int_{-1}^1 f(y) w_H(y|q) \mathcal{D}_{q,x} \left[\Sum H_n(x|q)H_n(y|q)
 \frac{q^{an/2}}{(q;q)_n}\right] dy\\
 =  \frac{2}{1-q}  \int_{-1}^1 w_H(y|q) \sum_{n=1}^\infty 
  H_{n-1}(x|q) H_n(y|q) \frac{q^{an/2}}{(q;q)_{n-1}}q^{(1-n)/2}\;
  f(y)\;  dy\\
  =   \frac{2q^{a/2}}{1-q}  \int_{-1}^1 \Sum H_n(x|q) 
   \frac{q^{an/2}}{(q;q)_{n}} q^{-n/2} w_H(y|q) 
  H_{n+1}(y|q) \; f(y)\;  dy\\
  = - q^{a/2} \int_{-1}^1  f(y)\; 
  \mathcal{D}_{q,y}\left[w_H(y|q)  \Sum H_n(x|q) H_n(y|q) 
  \frac{q^{an/2}}{(q;q)_n}\right]\; dy. 
 \eg
 \eea
 We then apply the $q$-integration by parts formula of 
 \cite{Bro:Eva:Ism}, see  \cite[Theorem 16.1.1]{Ismbook},  and obtain 
 \bea
 \notag
  (\D F_af)(x) = q^{a/2} (F_a(\D f))(x),
 \eea
 which proves \eqref{eq10.3}. 
 
By  direct computation, since with $y = \cos \f$,  we have  
 \begin{eqnarray}
 \notag
 \bg
 (F_a \D f)(\cos \t) = (q^a;q)_\infty  
 \\
 \times \int_0^\pi   f(\cos \f) 
 \mathcal{D}_{q,y}  \left[
 \frac{ w_H(\cos \f|q) }
 { (q^{a/2}e^{i(\t+\f)}, q^{a/2}e^{i(\t-\f)},  q^{a/2} e^{-i(\t+\f)}, 
 q^{a/2}e^{i(\f-\t)};q)_\infty} \right]  \;  dy \\
 =  \frac{2}{1-q} \; (q^a;q)_\infty \\
 \times \int_0^\pi    
 \frac{ (2y-2xq^{(a-1)/2})\; w_H(y|q)f(y)   dy}
{ (q^{(a-1)/2}e^{i\t+\f)}, q^{(a-1)/2}e^{i(\t-\f)},  q^{(a-1)/2} e^{-i(\t+\f)}, 
q^{(a-1)/2} e^{i(\f-\t)};q)_\infty}.  
 \eg
 \end{eqnarray}
This proves \eqref{eq10.4}. Finally \eqref{eq10.3} shows that if 
$(\D f)(x) (\D F_a f)(x) \ge 0$ hence the last assertion follow.  
 \end{proof}
 
 We note that the last statement in Theorem \ref{thm9.3}, 
 namely $(\D F_af)(x) \ge 0$ if $\D f \ge 0$, is an analogue of a monotonicity preserving property of $F_a$.

We now define the infinitesimal generator of the semigroup $F_a$. The \cqHp  form a basis 
for $L_2[-1,1,w_H]$, hence it suffices to define the infinitesimal generator $J$ on the basis. Part (d) of Theorem \ref{thm10.2}  
suggests that 
\begin{eqnarray}
\notag
J H_m(\cdot \vert q) = \lim_{a \to 0^+} \frac{q^{ma/2}-1}{a} H_m(x|q)= \frac{m}{2} \log q H_m(x|q).
\end{eqnarray} 
 Thus 
 \begin{eqnarray} 
\textup{If}\;  f = \Sum f_n H_n(\cdot \vert q), \quad \textup{then} \; (Jf)(x) := \frac{1}{2}\log q \;  
 \Sum n f_n H_n(x|q). 
\end{eqnarray}
 Thus $J$ is unbounded and is only densely defined on both $L_2[-1,1,w_H]$ and  $C[-1,1]$. 
 \begin{thm}
  Assume that $f$ is twice continuously differentiable. Then 
 \begin{eqnarray}
 \|F_a f - f\| = \O(a),
 \end{eqnarray}
 as $a \to 0^+$. 
 \end{thm}
 The proof is identical to the proof of Theorem \ref{thmOrderofApp}
 and uses Theorem \ref{thm10.1}.
 
We now provide an inversion formula for $F_a$ as operators defined 
on $L_2[-1,1,w_H]$. 
\begin{thm}\label{thm8.6}
Let $a\ge0$ and $g= \Sum q_n H_n(x|q)$. Then 
$g=F_a f$ with $f \in L_2[-1,1,w_H]$ if and only if  $g = \Sum g_n H_n(x|q)$ with $g_n q^{-na/2} \in \ell^2$. Furthermore, 
$f = \Sum g_n q^{-na/2} H_n(x|q)$. 
\end{thm}  
 The theorem follows from Parseval's theorem and the Riesz-Fischer theorem. 
  
 Another inversion formula follows from Theorem \ref{invTaSa}. Recall the notation, from \eqref{eqdefg},
 \bea
 g(\cos \t) = (-q^{1/4} e^{i\t}, -q^{1/4} e^{-i\t};q^{1/2})_\infty.
 \eea
 We then use the fact that 
 \bea
\frac{(1-q)^a}{2^aq^{a/4}} g F_a(f/g) = T_a f,
 \eea
 and apply part (a) of  
 Theorem \ref{invTaSa}.
 
 
 We now revisit \eqref{eqIntertwine}. We look for an operator $\mathcal{B}_q$ 
 such that $\mathcal{B}_q F_a$ is a constant multiple of $F_{a-1}$. This will then identify a multiple of $F_a$ as a fractional inverse operator to $\mathcal{B}_q$. 
  We found that  
 \bea
 (\mathcal{B}_qf)(x) = 2q^{1/2} \frac{\breve{f}(q^{1/2}z) - z^2 \breve{f}(q^{-1/2}z)} {(q-1)(z^2-1)},  
 \eea
 will do the job. It is well defined because it is symmetric in $z$ and $1/z$.  It is easy to see that 
 \bea 
 \notag
 \mathcal{B}_q F_a = \frac{2q^{1/4}}{1-q} F_{a-1}.
 \eea
Therefore we define the new semigroup $\{G_a: a >0\}$ by 
\bea
 G_a = \frac{(1-q)^a}{2^a q^{a/4}} F_a.
 \eea
 Therefore $G_a$ is a semigroup, and is a fractional analogue of 
 an inverse to $\mathcal{B}_q$, in the sense that  $\mathcal{B}_qG_a = G_{a-1}$. 
 
 We note that $\mathcal{B}_q$ also has the representation 
 \bea
 \label{eqdefB}
 (\mathcal{B}_q f)(x) = \frac{1}{g(x)}( \D f g)(x),
 \eea
 has the property 
 \
  This leads us to modify the definition of $F_a$ to 
 Now 
 \eqref{eqIntertwine} may be stated as 
 \bea
 \bg
\int_{-1}^1  \E(x;t) (G_af)(x) w_H(x|q)\; dx \qquad \qquad \\
\qquad \qquad \qquad\qquad 
  = \frac{(1-q)^a\; (q^{a+1}t^2;q^2)_\infty}{2^a q^{a/4}\; (qt^2;q^2)_\infty} \int_{-1}^1 \E(y,tq^{a/2}) f(x) w_H(y|q) dy.
\eg
\eea
 This is like a $q$-analogue of a multiplier for the transform whose kernel is $\E(x;t) w_H(x|q)$.

 \section{A $q$-Gauss--Weierstrass Transform} 
 The Gauss--Weierstrass transform $W(\l)$  is defined by 
 \begin{eqnarray}
 \label{eqGWT}
 (W(\l) f)(x) = \frac{1}{\sqrt{2\pi}} \int_\R e^{ (t-x)^2/2\l} f(t) dt,
 \end{eqnarray}
 \cite{Hir:Wid}. It is clear that we can rescale to make $\l=1$. 
 
 We now define the $q$-analogue of this transform by 
 \begin{eqnarray}
 \label{eqqGWT}
 (\mathcal{W}_qf)(t) =(qt^2;q)_\infty  \int_{-1}^1  \E (x;t) w_H(x|q) f(x) dx
 \end{eqnarray}
 The motivation for this is that $w_H$ with rescaling tends to $e^{-x^2/2}$, the weight function for Hermite polynomials, as $q \to 1^-$. Moreover $\E(x;t)$ tend to $e^{xt}$. Furthermore $(qt^2;q)_\infty$ tends to $e^{-t^2/2}$. There are several $q$-analogues of the exponential function, \cite{Gas:Rah},  and they all appear here. 
 
 Ismail and Zhang \cite{Ism:Zha1} proved that 
  \begin{eqnarray}
  \label{eqGFqH}
(qt^2;q)_\infty  \E (x;t)  = \Sum \frac{q^{n^2/4}t^n }{(q;q)_n}
 H_n(x|q).
 \end{eqnarray} 
This shows that $(\mathcal{W}_q e_0)(t)=1$. Note that \eqref{eqGFqH} is the analogue of the familiar generating function
\begin{eqnarray}
\notag
\Sum H_n(x) \frac{t^n}{n!} = \exp(2xt-t^2) = e^{-t^2} e^{2xt}.
\end{eqnarray}

We now invert this transform. 
For $f \in L_2[-1,1, w_h]$, let $f = \Sum f_n H_n(x|q)$. Then Parseval's formula implies 
\begin{eqnarray}
\label{eqWl2}
 (\mathcal{W}_qf)(t) = \Sum f_n q^{n^2/4} t^n.
 \end{eqnarray}
 Since $\{f_n\} \in \ell^2$ then  $(\mathcal{W}_qf)(t)$ is entire and of order zero. 
 Moreover 
 \begin{eqnarray}
 f_n q^{n^2/4} =\frac{1}{n!}\; \left. 
  \frac{d^n}{dx^n} (\mathcal{W}_qf)(t)\right|_{t=0}.
 \end{eqnarray}
 This proves the following theorem.
 \begin{thm}
 An entire function $g$ is a $q$-Gauss--Weierstrass transform of a
 a function in $L_2[-1,1, w_h]$ if and only if $g(x) = \Sum g_n x^n$, and $q^{-n^2/4}g_n \in \ell^2$. When this condition is satisfied 
 then  $f(x) = \Sum q^{-n^2/4} g_n H_n(x|q)$. 
 \end{thm}
 
 We note that an inversion of the Gauss--Weierstrass transform 
\eqref{eqGWT} also involve an operational representation, see 
\cite{Hir:Wid}.

\section{Dual Integral Equations}
If we let $x=\cos\t$ and  $y=\cos\f$, then we consider the following dual integral equations 
\begin{equation}
\left\{\begin{array}{ll}
\int_0^\pi  
\frac{(-q^{1/4} e^{i\t}, -q^{1/4} e^{-i\t};q^{1/2})_\infty w_H(\cos \f|q) \psi(\cos \f)}
{(-q^{1/4} e^{i\f}, -q^{1/4} e^{-i\f};q^{1/2})_\infty  h(\cos \f; q^{a/2}e^{i\t}, q^{a/2} e^{-i\t})} \sin \f d\f= F(x) & -1<x<0,\\
\int_0^\pi  
\frac{(-q^{1/4} e^{i\t}, -q^{1/4} e^{-i\t};q^{1/2})_\infty w_H(\cos \f|q) \psi(\cos \f)}
{(-q^{1/4} e^{i\f}, -q^{1/4} e^{-i\f};q^{1/2})_\infty  h(\cos \f; q^{b/2}e^{i\t}, q^{b/2} e^{-i\t})} \sin \f d\f=G(x) & 0<x<1.

\end{array}\right.
\end{equation}
where $a>0$ and $b>0.$
In terms of the  $T_a$ operator, the above equations are 
\begin{equation}\label{dual1}
\left\{\begin{array}{ll}
T_a \psi=(\frac{1-q}{2q^{1/4}})^a(q^a;q)_{\infty} F &  -1<x<0, \\
T_b \psi=(\frac{1-q}{2q^{1/4}})^b(q^b;q)_{\infty} G &   0<x<1.
\end{array}\right.
\end{equation}
Next, we follow a nice technique due to  Noble in \cite{Nob}. First we define two functions $f(x)$ and $g(x)$
\begin{equation}\label{solution}
T_a \psi=(\frac{1-q}{2q^{1/4}})^a(q^a;q)_{\infty} f,  \qquad T_b \psi=(\frac{1-q}{2q^{1/4}})^b(q^b;q)_{\infty} g,
\end{equation}
then it is easily to know $f(x)=F(x)$ on $(-1,0)$ and $g(x)=G(x)$ on $(0,1)$.
To simplify the writing we shall use the following notation 
\bea
\notag
\bg
I_1=(-1,0) \qquad \textup{and}\qquad  I_2=(0,1), \\
f_1=F \quad \text{in $I_1$}\quad f_1=0 \quad \text{in $I_2$,} \quad
f_2=0 \quad \text{in $I_1$}\quad f_2=f \quad \text{in $I_2$}, \\
g_1=g \quad \text{in $I_1$}\quad g_1=0 \quad \text{in $I_2$}, \quad
g_2=0 \quad \text{in $I_1$}\quad g_2=G \quad \text{in $I_2$}. 
\eg
\eea
Also, $f_1(x)=F(x)$ on $I_1$ and $g_2(x)=G(x)$ on $I_2$.
We discuss three cases according to $a>(<)b$, or $a =b$.

{\bf Case (a)}: Assume $a>b$,   then  part (a) of Theorem \ref{thm1}  
yields 
\begin{equation}
(\frac{1-q}{2q^{1/4}})^b(q^b;q)_{\infty}T_{a-b}g=(\frac{1-q}{2q^{1/4}})^a(q^a;q)_{\infty}f.
\end{equation}
The $f_1(x)$ and $g_2(x)$ are known, if we evaluate on interval $I_1$, then
\begin{equation}
(\frac{1-q}{2q^{1/4}})^b(q^b;q)_{\infty}T_{a-b}g_1=(\frac{1-q}{2q^{1/4}})^a(q^a;q)_{\infty}f_1-(\frac{1-q}{2q^{1/4}})^b(q^b;q)_{\infty}T_{a-b}g_2,
\end{equation}
which leads to a integral equations
\begin{equation}
\int_{-1}^0 g_1(y)K_1(x,y)dy=H(x) \quad \text{on $(-1,0)$},
\end{equation}
where
\begin{eqnarray}
\bg
K_1(x,y)=(\frac{1-q}{2q^{1/4}})^a(q^b;q)_{\infty} \times
(q^{(a-b)};q)_{\infty} \\
\frac{(-q^{1/4} e^{i\t}, -q^{1/4} e^{-i\t};q^{1/2})_\infty w_H(\cos \f|q) }
{(-q^{1/4} e^{i\f}, -q^{1/4} e^{-i\f};q^{1/2})_\infty  h(\cos \f; q^{(a-b)/2}e^{i\t}, q^{(a-b)/2} e^{-i\t})}.
\eg
\end{eqnarray}
and 
\begin{equation}
H(x)=(\frac{1-q}{2q^{1/4}})^a(q^a;q)_{\infty}f_1-(\frac{1-q}{2q^{1/4}})^b(q^b;q)_{\infty}T_{a-b}g_2.
\end{equation}
If we know $g_1(x)$, then we integrate over  $I_2$ and find that
\begin{equation}
(\frac{1-q}{2q^{1/4}})^b(q^b;q)_{\infty}T_{a-b}(g_1+g_2)=(\frac{1-q}{2q^{1/4}})^a(q^a;q)_{\infty}f_2.
\end{equation}
We are now able to evaluate $\psi(x)$ by Theorem  \ref{invTaSa} $\textup{(a)}$ and \eqref{solution}. 

{\bf Case (b)}: If $a=b$, Theorem  \ref{invTaSa} $\textup{(a)}$, \eqref{dual1} and \eqref{solution}
lead to 
\begin{equation}\
\psi(x)=\left\{\begin{array}{ll}
(\frac{1-q}{2q^{1/4}})^a(q^a;q)_{\infty} \mathcal{D}_q^{[a]+1}T_{1-\{a\}}f_1 &  -1<x<0, \\
(\frac{1-q}{2q^{1/4}})^b(q^b;q)_{\infty}\mathcal{D}_q^{[a]+1}T_{1-\{a\}} g_2 &   0<x<1.
\end{array}\right.
\end{equation}
Setting
\begin{equation}
h(x)=\left\{\begin{array}{ll}
(\frac{1-q}{2q^{1/4}})^a(q^a;q)_{\infty}f_1 &  -1<x<0, \\
(\frac{1-q}{2q^{1/4}})^a(q^a;q)_{\infty} g_2 &   0<x<1.
\end{array}\right.
\end{equation}
Then $\psi(x)$ can be written as 
\begin{eqnarray}
\bg
\psi(x)=(\frac{1-q}{2q^{1/4}})^{1-\{a\}}(q^{1-\{a\}};q)_{\infty}\\
\times \mathcal{D}_{q}^{[a]+1}\int_0^{\pi}\frac{(-q^{1/4} e^{i\t}, -q^{1/4} e^{-i\t};q^{1/2})_\infty w_H(\cos \f|q) h(\cos\f)\;  \sin \f d\f}
{(-q^{1/4} e^{i\f}, -q^{1/4} e^{-i\f};q^{1/2})_\infty  h(\cos \f; q^{(1-\{a\})/2}e^{i\t}, q^{(1-\{a\})/2} e^{-i\t})}.
\eg
\end{eqnarray} 
 
{\bf Case (c)}: If $a<b$, then we can see that 
\begin{equation}
(\frac{1-q}{2q^{1/4}})^b(q^b;q)_{\infty}g=(\frac{1-q}{2q^{1/4}})^a(q^a;q)_{\infty}T_{b-a}f.
\end{equation}
Next we evaluate the integral on the  interval  $I_2$ and find that 
\begin{equation}
(\frac{1-q}{2q^{1/4}})^a(q^a;q)_{\infty}T_{b-a}f_2=(\frac{1-q}{2q^{1/4}})^b(q^b;q)_{\infty}g_2-(\frac{1-q}{2q^{1/4}})^a(q^a;q)_{\infty}T_{b-a}f_1.
\end{equation}
It also leads to the  integral equation
\begin{equation}
\int_{0}^1 f_2(y)K_2(x,y)dy=P(x) \quad \text{on $(0,1)$},
\end{equation}
where
\begin{eqnarray}
\bg
K_2(x,y)=(\frac{1-q}{2q^{1/4}})^b(q^a;q)_{\infty}  
(q^{(b-a)};q)_{\infty} \\	\frac{(-q^{1/4} e^{i\t}, -q^{1/4} e^{-i\t};q^{1/2})_\infty w_H(\cos \f|q) }
{(-q^{1/4} e^{i\f}, -q^{1/4} e^{-i\f};q^{1/2})_\infty  h(\cos \f; q^{(b-a)/2}e^{i\t}, q^{(b-a)/2} e^{-i\t})}, 
\eg
\end{eqnarray}
and 
\begin{equation}
P(x)=(\frac{1-q}{2q^{1/4}})^b(q^b;q)_{\infty}g_2-(\frac{1-q}{2q^{1/4}})^a(q^a;q)_{\infty}T_{b-a}f_1.
\end{equation}
We then evaluate the integral on the interval $I_1$,and conclude that
\begin{equation}
(\frac{1-q}{2q^{1/4}})^b(q^b;q)_{\infty}g_1=(\frac{1-q}{2q^{1/4}})^a(q^a;q)_{\infty}T_{b-a}(f_1+f_2).
\end{equation}
This enables us to find  $\psi(x)$ by using Theorem  \ref{invTaSa} $\textup{(a)}$  and \eqref{solution}.

If we modify Sneddon's method \cite{Sne}, then we also can solve the  special case when $a-b$ is an integer. If $a-b=0$, then it is case (b) that we discuss above.  

{\bf Case(d)} if $a>b$, then define 
\begin{equation}
m(x)=\left\{\begin{array}{ll}
(\frac{1-q}{2q^{1/4}})^a(q^a;q)_{\infty}\mathcal{D}_q^{a-b}f_1 &  -1<x<0, \\
(\frac{1-q}{2q^{1/4}})^b(q^b;q)_{\infty} g_2 &   0<x<1.
\end{array}\right.
\end{equation}
It is clear that $T_b\psi(x)=m(x)$, and using Theorem  \ref{invTaSa} $\textup{(a)}$, we obtain the following  integral representation
 for $\psi$, 
\begin{eqnarray}
\bg
\psi(x)=(\frac{1-q}{2q^{1/4}})^{1-\{b\}}(q^{1-\{b\}};q)_{\infty}\\
\times 
\mathcal{D}_{q}^{[b]+1}\int_0^{\pi}\frac{(-q^{1/4} e^{i\t}, -q^{1/4} e^{-i\t};q^{1/2})_\infty w_H(\cos \f|q) m(\cos\f)\; \sin \f d\f}
{(-q^{1/4} e^{i\f}, -q^{1/4} e^{-i\f};q^{1/2})_\infty  h(\cos \f; q^{(1-\{b\})/2}e^{i\t}, q^{(1-\{b\})/2} e^{-i\t})} .
\eg
\end{eqnarray}

{\bf Case(e)} if $a<b$, then define $n$ by 
\begin{equation}
n(x)=\left\{\begin{array}{ll}
(\frac{1-q}{2q^{1/4}})^a(q^a;q)_{\infty}f_1 &  -1<x<0, \\
(\frac{1-q}{2q^{1/4}})^b(q^b;q)_{\infty}\mathcal{D}_q^{b-a} g_2 &   0<x<1.
\end{array}\right.
\end{equation}
Thus  $T_a\psi(x)=n(x)$, and using Theorem  \ref{invTaSa} $\textup{(a)}$, we establish  the  integral representation 
\begin{eqnarray}
\bg
\psi(x)=(\frac{1-q}{2q^{1/4}})^{1-\{a\}}(q^{1-\{a\}};q)_{\infty}\\
\times \mathcal{D}_{q}^{[a]+1}\int_0^{\pi}\frac{(-q^{1/4} e^{i\t}, -q^{1/4} e^{-i\t};q^{1/2})_\infty w_H(\cos \f|q) n(\cos\f)\;  \sin \f d\f}
{(-q^{1/4} e^{i\f}, -q^{1/4} e^{-i\f};q^{1/2})_\infty  h(\cos \f; q^{(1-\{a\})/2}e^{i\t}, q^{(1-\{a\})/2} e^{-i\t})}.
\eg	
\end{eqnarray}

{\bf Acknowledgments} The first author wishes to thank Zeev Ditzian 
of the University of Alberta for many enlightening discussions during the preparation of this paper.

\bigskip

\noindent M. E. H. I, 
School of Mathematics and Computational Science\\Guilin University 
of Electronic Technology\\ Guilin, Guangxi 541004, P. R. China\\ 
and Department of Mathematics\\
University of Central Florida, Orlando, Florida 32816\\
email: ismail@math.ucf.edu
  \bigskip
   
\noindent R. Zhang,  School of Mathematics and Computational 
Science\\Guilin University of Electronic Technology\\ Guilin, 
Guangxi 541004, P. R. China.\\ 
email:  ruimingzhang8@qq.com

 \bigskip 
 \noindent K. Zhou,
School of Mathematics and Statistics\\ Central South University\\ Changsha, Hunan 410083, P.R. China.\\
email address: krzhou1999@knights.ucf.edu

 \end{document}